\documentclass[12pt]{amsart}

\usepackage{amsmath, amssymb, amsthm,graphicx}
\usepackage{bbm}
\usepackage{enumerate}
\usepackage{caption}
\usepackage{subcaption}

\newcommand{\R}{\mathbb{R}}
\renewcommand{\P}{\mathbb{P}}
\newcommand{\F}{\mathcal{F}}

\newcommand{\V}{\mathcal{V}}

\newcommand{\E}{\mathbb{E}}
\newcommand{\FF}{\mathbb{F}}
\newcommand{\1}{\mathbbm{1}}

\theoremstyle{definition}
\newtheorem{ex}{Example}[section]

\theoremstyle{plain}
\newtheorem{thm}{Theorem}[section]

\newtheorem{remark}[thm]{Remark}
\newtheorem{lemma}[thm]{Lemma}
\newtheorem{corollary}[thm]{Corollary}
\newtheorem{prop}[thm]{Proposition}

\begin{document}

\author{Camilo Hern\'andez}
\address{
IEOR Departament\\
Columbia University\\
NY, USA.
}
\email{camilo.hernandez@columbia.edu}
 
\author{Mauricio Junca}
\address{
Mathematics Department\\
Universidad de los Andes\\
Bogot\'a, Colombia.
}
\email{mj.junca20@uniandes.edu.co}

\author{Harold Moreno-Franco}
\address{
Laboratory of Stochastic Analysis and its Applications\\
National Research University Higher School of Economics \\
Moscow, Russia.
}
\email{hmoreno@hse.ru}


\title[A time of ruin constrained optimal dividend problem]{A time of ruin constrained optimal dividend problem for spectrally one-sided L\'evy processes}

\begin{abstract}
We introduce a longevity feature to the classical optimal dividend problem by adding a constraint on the time of ruin of the firm. We extend the results in \cite{HJ15}, now in context of one-sided L\'evy risk models. We consider de Finetti’s problem in both  scenarios with and without fix transaction costs, e.g. taxes. We also study the constrained analog to the so called Dual model. To characterize the solution to the aforementioned models we introduce the dual problem and show that the complementary slackness conditions are satisfied and therefore there is no duality gap. As a consequence the optimal value function can be obtained as the pointwise infimum of auxiliary value functions indexed by Lagrange multipliers. Finally, we illustrate our findings with a series of numerical examples.
\end{abstract}



\maketitle

\section{Introduction}\label{Int}
Proposed in 1957 by Bruno de Finetti \cite{Definetti}, the problem of finding the dividend payout strategy that maximizes the discounted expected payout throughout the life of an insurance company has been at the core of actuarial science and risk theory. An important element of this problem is how one chooses to model the process describing the reserves of the firm, $X$. The solution to de Finetti's problem has been given for the case $X$ is assumed to be a compound Poisson process with negative jumps and positive drift, commonly referred as the \emph{Cram\'er-Lundberg} model, where $X$ is a Brownian motion, and the sum of the previous two, \cite{Schmidli,asta,tak}. Nowadays, the case in which $X$ is assumed to be a spectrally negative L\'evy process is the most general set up for which the problem has been studied (references below). The case when $X$ is a spectrally positive is also considered in the literature and it is known as the Dual model, \cite{KyYa14}. This set up fits in the context of a company whose income depends on inventions or discoveries. Both settings make a strong use of properties of the underlying L\'evy measure and fluctuation theory of L\'evy processes which requires the study of the so-called \emph{scale functions}.

A common result in all these scenarios is that the optimal strategy, in the absence of transaction cost, corresponds to a \emph{barrier/reflection strategy}. In such strategy the reserves are reduced to the barrier level by paying out dividends. Nevertheless, in general the solution is not necessarily of this type. In \cite{azcuemuler2005} the first example for the \emph{Cram\'er-Lundberg} model with Gamma claim distribution for which no barrier strategy is optimal was presented. Today, it is well known that, in the spectrally negative case, barrier strategies solve the optimal dividend problem if the tail of the L\'evy measure is log-convex\footnote{A function $f$ is said to be log-concave [resp. log-convex] if $\log(f)$ is concave [resp. convex].}, see \cite{Loeffen10}. Now, for the spectrally positive case, the optimal strategy is always a barrier strategy, \cite{KyYa14}. In the presence of transaction cost, when $X$ is a spectrally negative L\'evy process and the L\'evy measure has a log-convex density, \cite{LoeffenTrans} shows that the optimal strategy is given by paying out dividends in such a way that the reserves are reduced to a certain level $b_-$ whenever they are above another level $b_+>b_-\geq0$. This strategy is known as \emph{single band strategy}. The same result holds for the Dual model, \cite{BayraktarImpdual}.

However, a missing element in the current set up had long been noticed. The longevity aspect of the firm remained as a separate problem, see \cite{schmidli2002} for a survey on this matter. Despite efforts to integrate both features, \cite{Hipp03,Jostein03,ThonAlbr,Grandits}, it was not until very recent that a successful solution to a model that actually accounts for the trade-off between performance and longevity was presented. In \cite{HJ15}, the authors considered de Finetti's problem in the setting of Cram\'er-Lundberg reserves with exponentially distributed jumps adding a constraint on the expected time of ruin of the firm. 

The main contribution of this article is to extend the results of \cite{HJ15} for different models. Namely, to the case in which the reserves are modeled by a spectrally negative L\'evy process with complete monotone L\'evy measure with and without transaction cost, and the Dual model. As an intermediate step we also show that scale functions of the spectrally negative L\'evy process with complete monotone measure are strictly log-concave in an unbounded interval.

This paper is organized as follows: In Section \ref{problem} we present the problem we want to solve and describe the strategy to solve it. In the Section \ref{scale} we review the main results in fluctuation theory of spectrally negative L\'evy processes. Section \ref{SecdeFinetti} presents the solution to the constrained dividend problem for de Finetti's model, first without transaction cost and later including transaction cost. The result of strict log-concavity of scale functions is also included in this section, Corollary \ref{strictlogconcave}. Section \ref{SecDual} presents the solution of the constrained problem for the Dual model. In the following section we illustrate our results throughout a series of numerical examples. We finalize this article with a section of conclusions and questions.

\section{Problem formulation}\label{problem}

Let $X$ be the process modeling the reserves of the firm. In the setting of this paper we will assume $X$ to be a \emph{spectrally one-sided L\'evy process}, i.e. spectrally negative [resp. positive] L\'evy processes which have neither monotone paths nor positive [negative] jumps. The above process is defined on the filtered probability space $(\Omega,\F,\FF,\P)$, where  $\FF=(\F_t)_{t\geq0}$ is the natural filtration generated by the process $X$. Given the process $X$, we consider the family of probability measures $\{\P_x:x\in\R\}$ such that under $\P_x$ we have $X_0=x$ a.s. (and so $\P_0=\P$), and we denote by $\E_x$ expectation with respect to $\P_x$.

The insurance company is allowed to pay dividends which are modeled by the process $D=(D_t)_{t\geq0}$ representing the cumulative payments up to time $t$. A dividend process is called admissible if it is a non-decreasing, right continuous with left limits, i.e. c\`adl\`ag, process adapted to the filtration $\FF$ which starts at 0. Therefore, the reserves process under dividend process $D$ reads as
\begin{equation}\label{surplus}
L_t^D= X_t-D_t.
\end{equation}
Let $\tau^D$ denote the time of ruin under dividend process $D$, i.e., $\tau^D=\inf\{t\geq0: L_t^D<0\}$. We also require that the dividend process do not lead to ruin, i.e., $D_{t+}-D_t\leq L_t^D$ for $t<\tau^D$ and $D_t= D_{\tau^D}$ for $t\geq \tau^D$, so no dividends are paid after ruin. We call $\Theta$ the set of such processes. As proposed by de Finetti, the company wants to maximize the expected value of the discounted flow of dividend payments along its lifespan, where the lifespan of the company will be determined by its ruin. If we also consider a transaction cost each time dividends are paid, then a continuous dividend process is forbidden, therefore, we require in addition that the dividend processes $D$ are pure jump processes. So, the objective function of the company can be written as
\begin{align}
\V^D(x):=\E_{x}\left[ \int_0^{\tau^D-}e^{-q t}(d D_t-\beta d N^D_{t})\right],
\end{align}
where $q$ is the discount factor, $\beta\geq0$ the transaction cost and $N^D$ is the stochastic process that counts the number of jumps of $D$.

The purpose of this paper is to add a restriction on the dividend process $D$ to the previous problem, which we model by the constraint:
\begin{equation}\label{Rest}
\E_{x} \Big[e^{-q \tau^D}\Big]\leq K, \quad 0\leq K\leq1 \text{ fixed.}
\end{equation}
The motivation behind such a constraint is that it takes into account the time of ruin under the dividend process. One possible way to choose the parameter $K$ is to consider the equivalent constraint
$$\E_{x}\left[\int_0^{\tau^D}e^{-qt}dt\right]\geq\int_0^Te^{-qt}dt, \quad T>0,$$
as in \cite{HJ15}. Also, note that 
$$\E_{x} \Big[e^{-q \tau^D}\1_{\tau^D<\infty}\Big]\leq\E_{x} \Big[\1_{\tau^D<\infty}\Big]=\P_x(\tau^D<\infty),$$
hence, another possibility is to interpret the constraint as a restriction in the probability of ruin weighted by the time of ruin.

The advantage of the chosen form of the constraint, as it will be clear in the following sections, is that it fits in with the model in a smooth way. Combining all the above components we state the problem we aim to solve:
\begin{align*}\label{P1}
\tag{P}
V(x):=\underset{D\in \Theta}\sup\quad \V^D(x), \quad \text{s.t.} \quad \E_{x} \Big[e^{-q \tau^D}\Big]\leq K.
\end{align*}

In order to solve this problem we use Lagrange multipliers to reformulate our problem. For $ \Lambda \geq 0$ we define the function
\begin{equation}\label{lagrangian}
\V_{\Lambda}^{D}(x):=\V^D(x)-\Lambda\E_{x}\Big[e^{-q \tau^D }\Big]+ \Lambda K.
\end{equation}

We will follow the same strategy as in \cite{HJ15} to verify strong duality which is summarized here: First note that \eqref{P1} is equivalent to $\underset{D\in \Theta}\sup\,\, \underset{\Lambda\geq 0}\inf\,\,\V_{\Lambda}^{D}(x)$ since
$$\underset{\Lambda\geq 0}\inf\,\,\V_{\Lambda}^{D}(x)=\begin{cases} \V^D(x), &\mbox{if }\E_{x} \Big[e^{-q\tau^D}\Big]\leq K \\ 
-\infty, & \mbox{otherwise }. \end{cases} $$
Next, the dual problem of \eqref{P1}, is defined as 
\begin{equation}\label{D}
\tag{D}
\underset{\Lambda\geq 0}\inf\,\,\underset{D\in \Theta}\sup\,\, \V_{\Lambda}^{D}(x),
\end{equation}
which is always an upper bound for the primal \eqref{P1}. Therefore, the main goal of this paper is to prove that 
$$\underset{D\in \Theta}\sup\,\, \underset{\Lambda\geq 0}\inf\,\,\V_{\Lambda}^{D}(x)= \underset{\Lambda\geq 0}\inf\,\,\underset{D\in \Theta}\sup\,\, \V_{\Lambda}^{D}(x).$$ 

Now, to solve \eqref{D}, we can focus on solving for fixed $\Lambda\geq0$ the problem
\begin{equation}\label{P2}
\tag{P$_\Lambda$}
V_\Lambda(x):=\underset{D\in \Theta}\sup\,\, \V_{\Lambda}^{D}(x).
\end{equation}

Note that this is the optimal dividend problem with a particular type of Gerber-Shiu penalty function as in \cite{avram2015}. There, the authors considered the spectrally negative case under sufficient conditions on the L\'evy measure and prove the optimality of barrier and single band strategies without and with transaction cost, respectively. The optimal strategy in the Dual model, when $\beta=0$, also corresponds to a barrier strategy regardless of the L\'evy measure as shown in \cite{Yin}. In both scenarios the value of such barrier depends on the shape of the well known scale functions. Not surprisingly such family of functions is also the tool to characterize the solution of \eqref{P1}. The following section formally presents such family of functions and motivates its introduction in this context.

\section{Scale functions of Spectrally Negative L\'evy processes}\label{scale}

In this section $X$ will be assumed to be spectrally negative. However, the differences with the spectrally positive case should be clear since $-X$ is spectrally positive. For the process $X$ its \emph{Laplace exponent} is given by
\begin{equation}
\psi(\theta):= \log (\E[e^{\theta X_1}]),
\end{equation}
and it is well defined for $\theta\geq0$. The L\'evy-Khintchine formula guarantees the existence of a unique triplet $(\gamma,\sigma,\nu)$, with $\gamma\in \R$, $\sigma\geq0$ and $\nu$ a measure concentrated on $(-\infty,0)$ satisfying  $\int_{(-\infty,0)} (1\wedge x^2)\nu(dx)<\infty$, such that,
\begin{align*}
\psi(\theta)=\gamma \theta +\frac{1}{2}\sigma^2 \theta^2 +\int_{(-\infty,0)}(e^{\theta x}-1-\theta x \1_{\{-1<x\}})\nu(dx),
\end{align*}
for every $\theta\geq 0$. The triplet $(\gamma,\sigma,\nu)$ is commonly referred as the \emph{L\'evy triplet}.

Scale functions appear naturally in the context of fluctuation theory of spectrally negative L\'evy processes. More specifically, they are characterized as the family of functions $W^{(q)}:\R \rightarrow [0,\infty)$ defined for each $q\geq 0$, such that $W^{(q)}(x)=0$ for $x<0$ and it is the unique strictly increasing and continuous function whose Laplace transform satisfies
\begin{equation}\label{wqlaplace}
\int_0^\infty e^{-\beta x} W^{(q)}(x) dx= \frac{1}{\psi(\beta)-q}, 	\qquad  \beta >\Phi(q),
\end{equation}
where $\Phi(q):=\sup\{\theta\geq0: \psi(\theta)=q\}$ is the right inverse of $\psi(\theta)$. Such functions $W^{(q)}$ are referred as the \emph{q-scale functions}. Associated to this functions, we define for $q\geq0$  the functions $Z^{(q)}:\R\rightarrow[1,\infty)$ and $\bar{Z}^{(q)}:\R\rightarrow \R$ as 
\begin{align*}
Z^{(q)}(x)&:=1+q\int_0^x W^{(q)}(z)dz,\\
\bar{Z}^{(q)}(y)&:=\int_0^y Z^{(q)}(z)dz=y+q\int_0^y\int_0^z W^{(q)}(w)dw dz.
\end{align*}

We now review some properties of the scale functions, available for example in \cite{KKRivero2013}, that will be needed later on. First, it is useful to understand their behaviour at $0$ and at $\infty$. For $q\geq 0$, $W^{(q)}(0)=0$ if and only if X has unbounded variation. Otherwise, $W^{(q)}(0)=1/c$, where $c=\gamma+\int_{-1}^0 |x| \nu(dx)$. Recall that $c$ must be strictly positive to exclude the case of monotone paths. The initial value of the derivative of the scale function is given by 
\begin{align*}
W^{(q)'}(0+)=\begin{cases}
2/\sigma^2,\qquad &\text{if } \sigma > 0 \\
(\nu(-\infty,0)+q)/c^2,\qquad &\text{if } \sigma=0 \text{ and } \nu(-\infty,0)<\infty\\
\infty, & \text{otherwise}.
\end{cases}
\end{align*}

Regarding the behavior at infinity, we know that
\begin{align}\label{limitqfact}
\lim_{x\rightarrow \infty} e^{-\Phi(q)x}W^{(q)}(x)=\frac{1}{\psi'(\Phi(q))}.
\end{align}

\begin{remark}\label{logconcaveprop}
From \cite{Loeffen08} we know that  $q$-scale functions are always log-concave on $(0,\infty)$.
\end{remark}

An useful representation of scale functions was provided in \cite{Loeffen08}. Making use of Bernstein's theorem it was proven that when the L\'evy measure $\nu$ has a completely monotone density\footnote{A function $f$ is said to be \emph{completely monotone} if $f \in C [0,\infty)$, $f \in C^\infty (0,\infty)$ and satisfies $(-1)^n \frac{d^n}{dx^n}f(x)\geq 0$}, and $q>0$, 
 \begin{align}\label{qscalecompmon}
 W^{(q)}(x)=\frac{e^{\Phi(q)x}}{\psi'(\Phi(q))}- f(x), \quad x>0, 
 \end{align}
 with $f$ a completely monotone function. Furthermore, from the proof of this result it is known that $f(x)=\int_{0+}^{\infty}e^{-xt}\xi(dt+\Phi(q))$ where $\xi$ is a finite measure on $(0,\infty)$. Using this one can deduce that $f^{(n)}(x)\rightarrow0$ as $x\rightarrow\infty$ for all non-negative integers $n$. Also, that $q$-scale functions are infinitely differentiable, and odd derivatives are strictly positive and strictly log-convex.

Scale functions are present in a vast majority of fluctuation identities of spectrally negative L\'evy processes and, as we will see next, they appear in the setting of the optimal dividend problem. 

\section{Solution of the constrained de Finetti's problem}\label{SecdeFinetti}

Let us consider the case where the reserves process $X$ is a spectrally negative L\'evy process. We will solve the constrained problem in both scenarios, first without transaction cost, and then for $\beta>0$.

\subsection{No transaction cost}

As mentioned before, optimal strategies for Problem \eqref{P2} in this setting are barrier strategies. If we consider the dividend barrier strategy at level $b$, $D^b$, we have that $D_t^{b}=(b\vee \overline{X}_t)-b$ for $t\geq 0$, where $\overline{X}_t:=\underset{0\leq s\leq t}{\sup} X_s$ and therefore $X_t^{D^b}=b - [(b\vee \overline{X}_t)-X_t]$. The process in square brackets is a type of reflected process. More generally, for a given process $Y$ we define $\hat{Y}_t^s:=s\vee \overline{Y}_t-Y_t, t\geq0$, known as the \emph{reflected process at its supremum with initial value $s$}. For such processes also define the exit time $\hat{\sigma}_k^s:=\inf \{t>0:\hat{Y}_t^s>k\}$. From the previous definitions it follows that $X_t^{D^b}=b-\hat{X}_t^{b}$ and $\tau^{D^{b}}=\hat{\sigma}_{b}^{b}$. This simple observation provides an useful identity for the value function when a barrier strategy is followed. The next identity, first presented in \cite{Gerber72}, can be found in \cite{kyprianou2014}.

\begin{prop}
Let $b>0$ and consider the dividend process $D_t^{b}=X_t -(b- \hat{X}_t)$. For $x\in [0,b]$,
\begin{equation}\label{Valuefunctqscale}
\V^{D^b}(x)=\E_x\left[\int_0^{\tau^{D^{b}}-} e^{-qt}dD_t^b\right]=\frac{W^{(q)}(x)}{W^{(q)'}_{+}(b)},
\end{equation}
where $W^{(q)'}_{+}(b)$ is understood as the right derivative of $W^{(q)}$ at $b$.
\end{prop}

The previous proposition suggests that in the unconstrained de Finetti's problem, the existence of an optimal barrier strategy boils down to the existence of a minimizer of $W^{(q)'}$, hence the importance of understanding the properties of scale functions and its derivatives. It was first shown in \cite{Loeffen082} that when $W^{(q)}$ is sufficiently smooth, meaning it is once [resp. twice] continuously differentiable when X is of bounded [resp. unbounded] variation, and $W^{(q)'}$ is increasing on $(b^*,\infty)$, where $b^*$ is the largest point where $W^{(q)'}$ attains its minimum, the barrier strategy al level $b^*$ is optimal. A sufficient condition for such properties to be satisfied is for $X$ to have a L\'evy measure with complete monotone density. This work also showed that $W^{(q)'}$ is strictly convex on $(0,\infty)$. Later results in \cite{KyprianouRS10}, stated that under the weaker assumption of a log-convex density of the L\'evy measure the same result holds on $(b^*,\infty)$ but not necessarily on $(0,\infty)$. Finally, \cite{Loeffen10} made a final improvement showing that if the tail of the L\'evy measure is log-convex the scale function of the spectrally negative L\'evy process has a log-convex derivative.

\subsubsection{Solution of \eqref{P2}}\label{dualclassical}

Likewise, the role of the scale functions in the setting of the constrained dividend problem is very important. The next result follows from \cite{Avram04} and it is also shown in \cite{avram2015}.

\begin{prop}\label{Lagrangianbarrier}
For a sufficiently smooth $q$-scale function $W^{(q)}$, the function $\V_\Lambda^{D^b}$, where $D^b$ is the barrier strategy at level $b \geq 0$, for $x\geq 0$ is given by
\begin{align}\label{lagrangianbarrier}
\V_\Lambda^{D^b}(x)=
\begin{cases}
W^{(q)}(x)\Big[\frac{1+q\Lambda W^{(q)}(b)}{W^{(q)'}(b)}\Big]-\Lambda Z^{(q)}(x) + \Lambda K &\text{if}\quad x\leq b\\
x-b+\mathcal{V}_\Lambda^{D^b}(b) &\text{if}\quad x>b. 
\end{cases}
\end{align}
\end{prop}

The solution of \eqref{P2} can be extracted from \cite{Loeffen08}. In that work it was proven that barrier strategies are optimal under the assumption of complete monotonicity of the L\'evy measure, so in this section we will make this assumption. To certify the optimality of an admissible barrier strategy two steps are carried out: First use \eqref{lagrangianbarrier} to propose a candidate for optimal barrier level, and second certify optimality with a verification lemma argument. To understand the solution of \eqref{P2}, we will elaborate on how such candidate is proposed.

In light of Proposition \ref{Lagrangianbarrier}, define the function $\zeta_{\Lambda}:[0,\infty)\rightarrow \R$ by
\begin{align}\label{functionZ}
\zeta_{\Lambda}(\varsigma):=\frac{1+q\Lambda W^{(q)}(\varsigma)}{W^{(q)'}(\varsigma)},\quad \varsigma>0
\end{align}
and $\zeta_{\Lambda}(0):=\underset{\varsigma\downarrow 0}{\lim}\, \zeta_{\Lambda}(\varsigma)$. Now, the barrier strategy at level
\begin{align}\label{bLambda}
b_\Lambda:=\sup \{b:\zeta_{\Lambda}(b)\geq \zeta_{\Lambda}(\varsigma),\text{ for all } \varsigma\geq 0\} 
\end{align}
is proposed as candidate optimal strategy for \eqref{P2}. 
\begin{remark}\label{complmonotprop}
From \cite{Loeffen08} we know that when the L\'evy measure $\nu$ is assumed to have completely monotone density, the set of maxima of the function $\zeta_{\Lambda}$ consists of a single point for all $\Lambda$. In fact, $\zeta_{\Lambda}$ is strictly increasing in $(0,b_{\Lambda})$ and strictly decreasing in $(b_{\Lambda},\infty)$ .
\end{remark}

We now state the theorem that characterizes the solution of \eqref{P2}. 
\begin{thm}[Optimal strategy for \eqref{P2}]
Suppose the L\'evy measure of the spectrally negative L\'evy process $X$ has a completely monotone density. Then the optimal strategy consists of a barrier strategy at level $b_\Lambda$ given by \eqref{bLambda}, and the corresponding value function is given by equation \eqref{lagrangianbarrier}.
\end{thm}

\subsubsection{Solution of \eqref{P1}}\label{P1classical}

We now proceed to solve \eqref{P1} following the same ideas as in \cite{HJ15}. Let $b_0$ be the optimal barrier for \eqref{P2} with $\Lambda=0$, that is, the optimal barrier for the unconstrained problem. Let  $\bar{\Lambda}:=\sup \{\Lambda\geq 0: b_\Lambda=0\}\vee0$. Note that if $\bar{\Lambda}>0$, then $b_0=0$. Now, since $b_{\Lambda}$ is the only maximum of $\zeta_{\Lambda}$, we consider the function $\Lambda:[b_0,\infty)\rightarrow \R_+$ defined by
\begin{align}\label{btolambdamap}
\Lambda(b):=\begin{cases} 
0 & \mbox{if } b=b_0\\
\frac{-W^{(q)''}(b)}{q[W^{(q)}(b) W^{(q)''}(b)-[W^{(q)'}(b)]^2]}  &   \mbox{if } b>b_0.
\end{cases}
\end{align}

We will show that this function establishes a biyection between $b$ and $\Lambda$ such that $b_{\Lambda(b)}=b$. Figure \ref{graphLambda} shows the behavior of the map when $\bar{\Lambda}=0$ and $\bar{\Lambda}>0$. 

\begin{figure}[t]
      \includegraphics[width=0.46\linewidth]{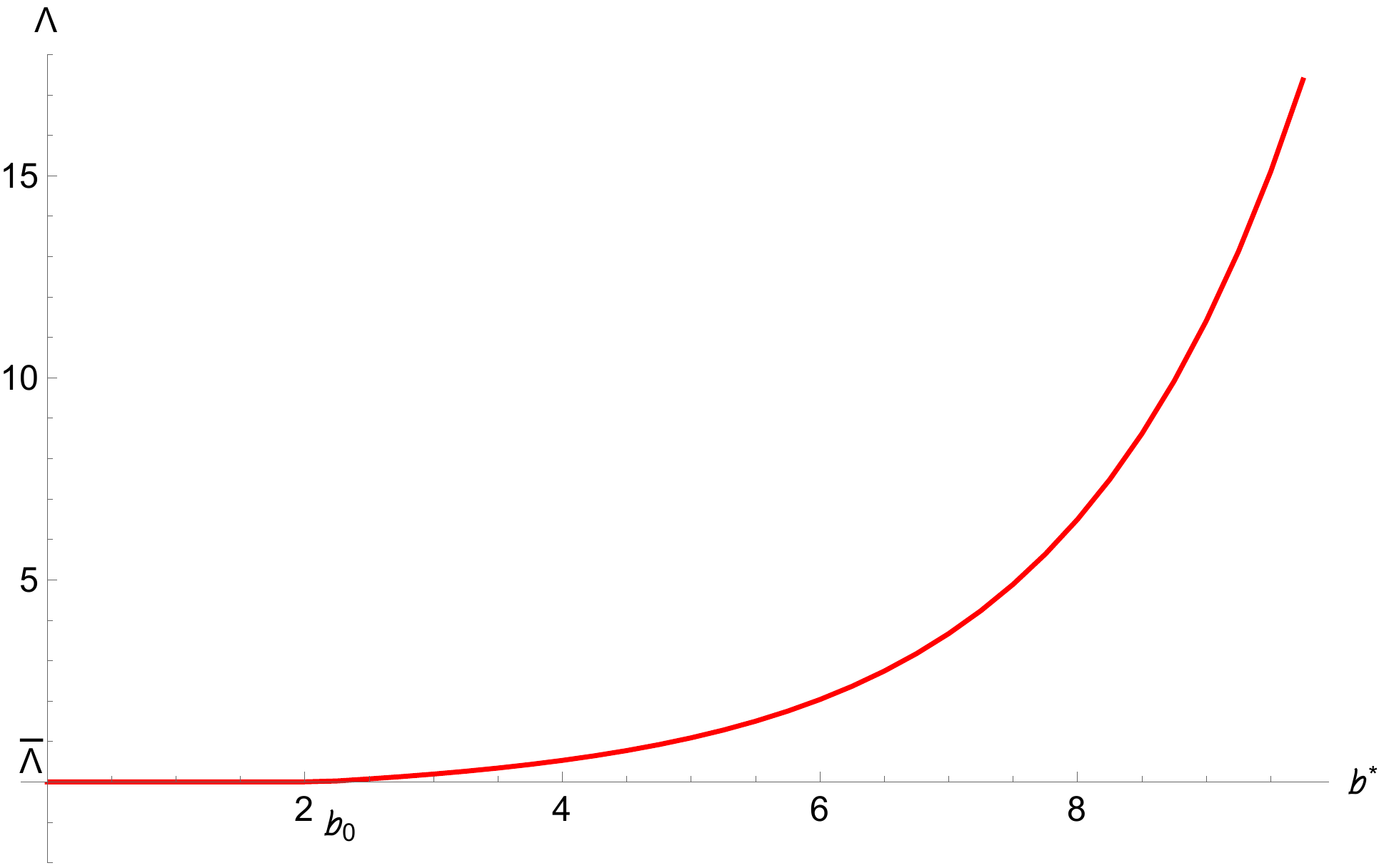}
\hfill
      \includegraphics[width=0.46\linewidth]{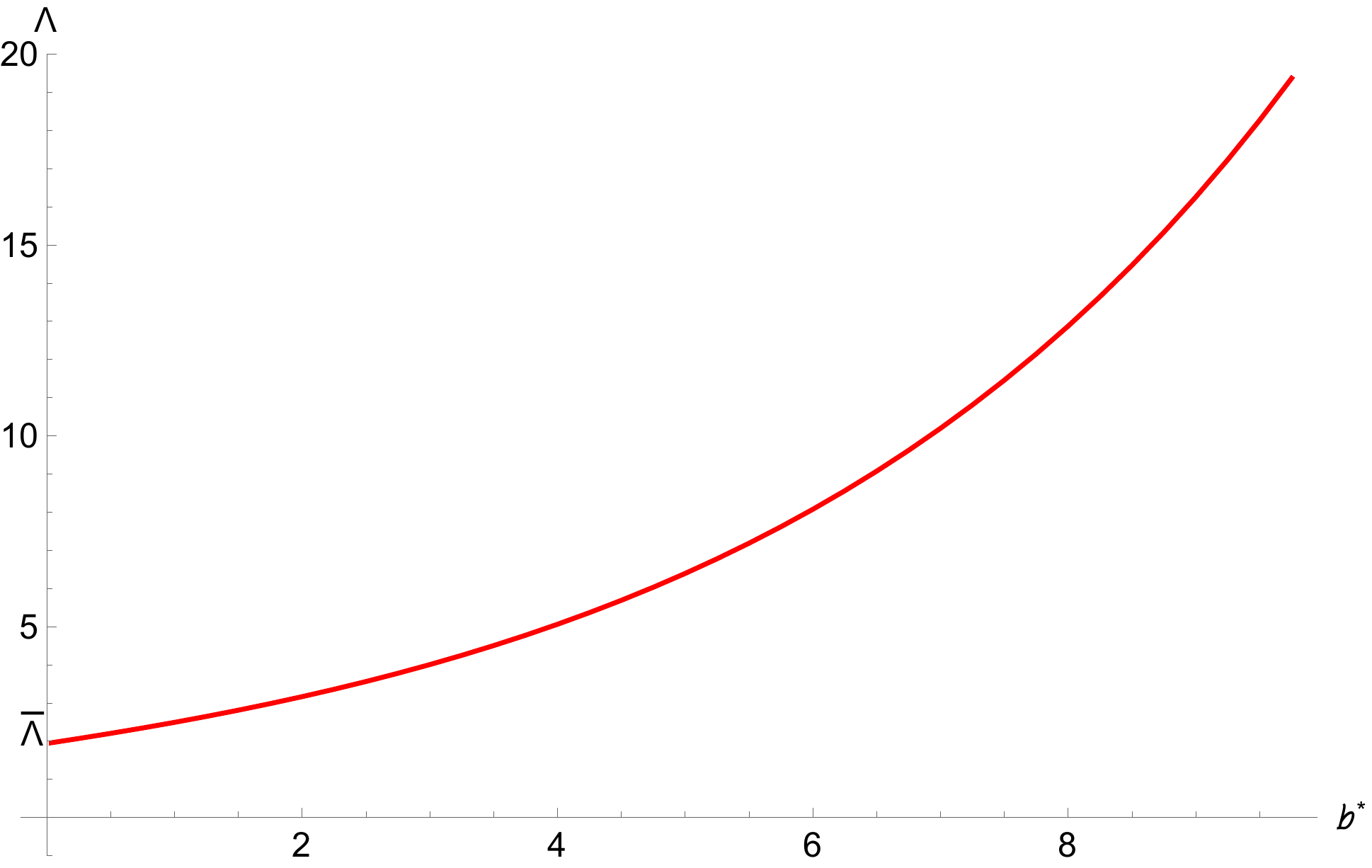}
      \caption{The map $\Lambda(b)$. On the left $\bar{\Lambda}=0$ and $b_0>0$. On the right $\bar{\Lambda}>0$ and $b_0=0$. These maps correspond to different choices of parameters for the Cram\'er-Lundberg model with exponential claims. See \cite{HJ15} for the explicit formula of the map.}\label{graphLambda}
\end{figure}

\begin{prop}\label{btolambdaprop}
For each $b \in (b_0,\infty)$ the barrier strategy at level $b$ is optimal for \eqref{P2} with $\Lambda(b)$. Also, this map is strictly increasing.
\begin{proof}
We want to show that the function $\Lambda(b)$ is well defined and maps barrier levels $b>b_0$ to $\Lambda(b)$ such that the pair $(D^b,\Lambda(b))$ is optimal for \eqref{P2}. To see this, first  recall from \cite{Loeffen08} that $W^{(q)''}(b)$ is strictly positive for $b>b_0$. Also, the log-concavity of the $q$-scale function implies that $W^{(q)}(b) W^{(q)''}(b)-[W^{(q)'}(b)]^2\leq0$ and we claim it cannot be 0. Since
\begin{align*}
\zeta'_{\Lambda}(\varsigma)=-\frac{W^{(q)''}(\varsigma)+\Lambda q[W^{(q)}(\varsigma) W^{(q)''}(\varsigma)-[W^{(q)'}(\varsigma)]^2]}{[W^{(q)'}(\varsigma)]^2},
\end{align*}
this would prove that $\zeta'_{\Lambda(b)}(b)=0$ for $b>b_0$ and therefore the pair $(D^b,\Lambda(b))$ is optimal for \eqref{P2}. To prove the claim we argue by contradiction. Let $\hat{b}>b_0$ be the minimum such that $W^{(q)}(\hat{b}) W^{(q)''}(\hat{b})-[W^{(q)'}(\hat{b})]^2=0$. We have two possibilities: Either $W^{(q)}(b') W^{(q)''}(b')-[W^{(q)'}(b')]^2<0$ for a some value $b'>\hat{b}$, or the expression equals zero in $[\hat{b},\infty)$. In the first case, by the continuity of $\Lambda(\cdot)$ in its domain, we will have two values $b''<\hat{b}<b'$ such that $\Lambda(b'')=\Lambda(b')$. This implies that those two barrier values in $(b_0,\infty)$ are optimal for \eqref{P2} for the same value of $\Lambda$, which contradicts Remark \ref{complmonotprop}. In the later case, it follows that the tail of $\log(W^{(q)})$ is linear, and so is the tail of $\log(W^{(q)'})$, which contradicts the strict log-convexity of $W'$ on $(0,\infty)$, see Remark \ref{complmonotprop}. Finally, as $W^{(q)'}(x)$ is strictly log-convex and strictly positive on $(0,\infty)$,
\begin{align*}
\frac{d\Lambda(b)}{db}=\frac{W^{(q)'}(b)[W^{(q)'}(b)W^{(q)'''}(b)-[W^{(q)''}(b)]^2]}{q[W^{(q)}(b) W^{(q)''}(b)-[W^{(q)'}(b)]^2]^2}
\end{align*}
is always positive and so the map is strictly increasing.
\end{proof}
\end{prop}

As a consequence of the proof of the previous proposition we have the following important property of $q$-scale functions.

\begin{corollary}\label{strictlogconcave}
$W^{(q)}(x)$ is strictly log-concave in $(b_0,\infty)$.
\end{corollary}

The behavior of the map $\Lambda(b)$ at infinity will be important for the final result of this section. 
\begin{lemma}\label{lambdainftylemma}
$\Lambda(b)\rightarrow \infty$ as $b\rightarrow \infty$.
\end{lemma}
\begin{proof}Note that
\begin{align*}
\Lambda(b)=\frac{1}q  \left[\frac{W^{(q)'}(b)^2}{W^{(q)''}(b)}-W^{(q)}(b)\right]^{-1},
\end{align*}
so, in order to prove the result we need to show that the term in brackets goes to 0. Using \eqref{qscalecompmon} we can obtain the following:
\begin{align*}
\frac{W^{(q)'}(b)^2}{W^{(q)''}(b)}-W^{(q)}(b)=&\frac{\left[\frac{\Phi(q)e^{\Phi(q)b}}{\psi'(\Phi(q))}- f'(b)\right]^2}{\left[\frac{\Phi(q)^2 e^{\Phi(q)b}}{\psi'(\Phi(q))}- f''(b)\right]}-\left[\frac{e^{\Phi(q)b}}{\psi'(\Phi(q))}- f(b)\right]\\
=& \frac{e^{\Phi(q)b}}{\psi'(\Phi(q))}\left[ \frac{\frac{\Phi(q)^2 e^{\Phi(q)b}}{\psi'(\Phi(q))}-\Phi(q)f'(b)}{\frac{\Phi(q)^2 e^{\Phi(q)b}}{\psi'(\Phi(q))}- f''(b)} -1\right]\\
&\hspace{1cm}+f(b)-f'(b)\left[\frac{\frac{\Phi(q)e^{\Phi(q)b}}{\psi'(\Phi(q))}- f'(b)}{\frac{\Phi(q)^2 e^{\Phi(q)b}}{\psi'(\Phi(q))}- f''(b)}\right].
\end{align*}
Since $f$ and all its derivatives vanish at infinity, we have that for $b$ large
\begin{align*}
\frac{W^{(q)'}(b)^2}{W^{(q)''}(b)}-W^{(q)}(b)=&\frac{e^{\Phi(q)b}}{\psi'(\Phi(q))}\left[ \frac{f''(b)-\Phi(q)f'(b)}{\frac{\Phi(q)^2 e^{\Phi(q)b}}{\psi'(\Phi(q))}- f''(b)}\right]+o(1)\\
=&\frac{1}{\psi'(\Phi(q))} \left[\frac{f''(b)-\Phi(q)f'(b)}{\frac{\Phi(q)^2 }{\psi'(\Phi(q))}- \frac{f''(b)}{e^{\Phi(q)b}}}\right]+o(1).
\end{align*}
Now, since
\begin{align*}
\frac{f''(b)-\Phi(q)f'(b)}{\frac{\Phi(q)^2 }{\psi'(\Phi(q))}- \frac{f''(b)}{e^{\Phi(q)b}}} \longrightarrow 0,\quad \textrm{as} \quad b\rightarrow\infty,
\end{align*}
we have the result.
\end{proof}

From the previous lemma and Proposition \ref{btolambdaprop} we obtain the following corollary.
\begin{corollary}\label{bLambdainfty}
The map $\Lambda(b)$ is one-to-one onto $(\bar{\Lambda},\infty)$. Furthermore, $b_{\Lambda}$ is strictly increasing and goes to $\infty$ as $\Lambda$ goes to $\infty$.
\end{corollary}

Now, in order the show the complementary slackness condition (condition \eqref{cond3clas} in the proposition below), we need to understand the behavior of the constraint as a function of the barrier level. Observing Equations \eqref{Valuefunctqscale} and \eqref{lagrangianbarrier}, we introduce the function 
\begin{equation}\label{Psi}
\varPsi_x(b):=\mathbb{E}_{x}\left[e^{-q \tau^{D^b}}\right]=
\begin{cases} 
Z^{(q)}(x)-q \frac{W^{(q)}(b)}{W^{(q)'}(b)}W^{(q)}(x) & \mbox{if } 0\leq x\leq b\\
\varPsi_b(b)  &   \mbox{if } x>b.
\end{cases}.
\end{equation}

\begin{prop}\label{optimalpair}
For each $x\geq0$ there exists $\bar{K}_{x}\geq0$ such that if $K>\bar{K}_{x}$, there exists $b^*$ which satisfies :
\begin{enumerate}[(i)]
\item\label{cond2clas} $\mathbb{E}_{x}\left[e^{-q \tau^{D^{b^*}}}\right]\leq K$ and
\item\label{cond3clas} $\Lambda(b^*)\left(K-\mathbb{E}_{x}\left[e^{-q \tau^{D^{b^*}}}\right]\right)=0$.
\end{enumerate}
\end{prop}
\begin{proof}
If $x\leq b$, $\varPsi_x(b)$ is given by \eqref{Psi}. Rewriting this expression as
$$\varPsi_x(b)=-q W^{(q)}(x)\Big[\frac{d \log (W^{(q)}(b))}{db}\Big]^{-1}+Z^{(q)}(x),$$
we can easily see that
$$\frac{d \varPsi_x(b)}{db}=qW^{(q)}(x)\frac{d^2\log(W^{(q)}(b))}{db^2}\Big[\frac{d \log (W^{(q)}(b))}{db}\Big]^{-2}<0,$$
for $b\in(b_0,\infty)$, from Corollary \ref{strictlogconcave}. Otherwise, if $x>b$, then  $\varPsi_x(b)= \varPsi_b(b)$ and some calculations yield that
$$\frac{d \varPsi_b(b)}{db}=qW^{(q)}(b)\frac{d^2\log(W^{(q)}(b))}{db^2}\Big[\frac{d \log (W^{(q)}(b))}{db}\Big]^{-2},$$
which is again strictly negative for $b\in(b_0,\infty)$. So, for fixed $x$, $\varPsi_x(b)$ is strictly decreasing as a function of $b$ in $(b_0,\infty)$, and just decreasing before $b_0$. Now, let 
\begin{equation}\label{Klimitclass}
\bar{K}_{x}:=\lim\limits_{b\rightarrow\infty}\varPsi_x(b)=-q\frac{W^{(q)}(x)}{\Phi(q)}+Z^{(q)}(x),
\end{equation}
where we use \eqref{limitqfact} to find the limit. Now, if $K\geq \varPsi_{x}(b_0)$, then the unconstrained problem satisfies the restriction and therefore $b^*=b_0$ satisfies the conditions. Otherwise, if $\bar{K}_x< K < \varPsi_{x}(b_0)$,  there exists $b^*>b_0$ such that $\varPsi_{x}(b^*)=K$, since $\varPsi_x(b)$ is strictly decreasing. This $b^*$ satisfies the conditions.
\end{proof}

\begin{remark}\label{remdonothing}
Note that $\bar{K}_{x}=\mathbb{E}_{x}\left[e^{-q \tau^0}\right]$, where $\tau^0$ is the time of ruin when no dividends are paid, see also \cite{Loeffen08}.
\end{remark}

The special case $K=\bar{K}_{x}$ requires the following lemma.
\begin{lemma}\label{limitclassical}
Let $x\geq0$. If $K=\bar{K}_{x}$ then $\Lambda(b)\Big(K-\mathbb{E}_x\Big[e^{-q \tau^{D^b}}\Big]\Big)\rightarrow 0$ as $b\rightarrow \infty$.
\begin{proof}
First, note that $\Lambda(b)\left(K-\mathbb{E}_x\Big[e^{-q \tau^{D^{b}}}\Big]\right)\leq 0$ for all $b> b_0$. Also, from \eqref{Valuefunctqscale} $\V^{D^b}(x)\rightarrow 0$ as $b$ goes to $\infty$. On the other hand, from the previous remark the do-nothing strategy is feasible for \eqref{P1} and hence $0\leq V(x)$. Finally, by weak duality we have that 
\begin{align*}\label{eqlimitclassical}\nonumber
0\leq V(x)&\leq\underset{\Lambda\geq 0}\inf\,\,V_{\Lambda}(x)\\\nonumber
&\leq \underset{b\to \infty}\lim V_{\Lambda(b)}(x)\\
&=\underset{b\to \infty}\lim\Lambda(b)\Big(K-\mathbb{E}_x\Big[e^{-q \tau^{D^b}}\Big]\Big)\leq0.
\end{align*}
\end{proof}
\end{lemma}
 
All this is enough to derive the main result.

\begin{thm}\label{strongduality}Let $x\geq 0$, $K\geq0$ and $V(x)$ be the value function of \eqref{P1}. Then $$V(x)\geq \underset{\Lambda\geq 0}\inf\,\,V_{\Lambda}(x)$$
 and therefore, $\underset{\Lambda\geq 0}\inf\,\,V_{\Lambda}(x)=V(x)$.
\end{thm}
\begin{proof}
Fix $x\geq 0$. We consider the following cases:
\begin{itemize}
\item \underline{$K>\bar{K}_{x}$}: By Proposition \ref{optimalpair} there is $b^*$ such that
\begin{align*}
\underset{\Lambda\geq 0}\inf\,\,V_{\Lambda}(x)&\leq V_{\Lambda(b^*)}(x)\\
&= \V^{D^{b^*}}- \Lambda(b^*) \mathbb{E}_{x}\Big[  e^{-q \tau^{D^{b^*}}} \Big] + \Lambda(b^*)K\\
&= \V^{D^{b^*}}\leq V(x),
\end{align*}
where the last inequality follows since the barrier strategy $D^{b^*}$ satisfies the constraint.
\item \underline{$K=\bar{K}_{x}$}: From the proof of Lemma \ref{limitclassical}, it follows that 
$$0=\underset{\Lambda\geq 0}\inf\,\,V_{\Lambda}(x)=V(x).$$
\item \underline{$K<\bar{K}_{x}$}: Here, we have that there exists $\epsilon>0$ such that $\mathbb{E}_{x}\left[e^{-q \tau^{D^b}}\right]>K+\epsilon$ for all $b$ . Hence $$\Lambda(b)\left(K-\mathbb{E}_x\left[e^{-q \tau^{D^b}}\right]\right)<-\Lambda(b)\epsilon.$$
Letting $b \rightarrow\infty$ we obtain that 
$\underset{\Lambda\geq 0}\inf\,\,V_{\Lambda}(x)=-\infty\leq V(x)$. Note that in this case \eqref{P1} is infeasible. 
\end{itemize}
\end{proof}

\subsection{With transaction cost}

We now consider the case where $\beta>0$. We will continue assuming that L\'evy measure of the spectrally negative process $X$ has a completely monotone density.  In this case we need to consider single band strategies for $b=(b_-,b_+)$ with $b_+>b_-\geq0$ denoted by $D^b$. Using the two-sided exit above fluctuation identity, \cite{LoeffenTrans} shows the following result.
\begin{prop}
Let $b$ be a single band strategy and consider the dividend process $D_t^{b}$ with $X$ a spectrally negative L\'evy process. The function $\V^{D^b}$ with transaction cost $\beta>0$, for $x\geq0$ is given by
\begin{equation}\label{ValuefunctqscaleTrans}
\V^{D^b}(x)=\begin{cases}
W^{(q)}(x)\dfrac{b_{+}-b_{-}-\beta}{W^{(q)}(b_{+})-W^{(q)}(b_{-})}, & \mbox{if }  x\leq b_{+}\\
x-b_{-}-\beta+\V^{D^b}(b_{-}), &  \mbox{if } x>b_{+},
\end{cases}
\end{equation}
\end{prop}

\subsubsection{{Solution of \eqref{P2}}}

Also, \cite{avram2015} shows the equivalent result for the function $\V_\Lambda^{D^b}$.
\begin{prop}\label{LagrangianbarrierTrans}
The function $\V_\Lambda^{D^b}$ with transaction cost $\beta>0$, where $D^b$ is the single band strategy  $b=(b_-,b_+)$, for $x\geq0$ is given by
\begin{equation}\label{lagrangianbarrierTrans}
\V_{\Lambda}^{D^b}(x)=\begin{cases} 
W^{(q)}(x)G_{\Lambda}(b_{-},b_{+})-\Lambda Z^{(q)}(x)+\Lambda K, & \mbox{if } x\leq b_{+}\\
x-b_{-}-\beta+\V_{\Lambda}^{D^b}(b_{-}), &  \mbox{if } x>b_{+}
\end{cases}
\end{equation}
where
\begin{equation}\label{functionG}
G_{\Lambda}(b_{-},b_{+}):=\frac{b_{+}-b_{-}-\beta+q\Lambda\int_{b_{-}}^{b_{+}}W^{(q)}(z)d z }{W^{(q)}(b_{+})-W^{(q)}(b_{-})}.
\end{equation}
\end{prop}

As expected, there is a close relation between \eqref{functionG} and the function $\zeta_{\Lambda}$ defined by \eqref{functionZ}.

\begin{remark}
If $\beta=0$ and letting $b_{-}\rightarrow b_{+}$ in \eqref{functionG}, we can see that 
\begin{align*}
 \lim_{b_{-}\rightarrow b_{+}}G_{\Lambda}(b_{-},b_{+})&=\lim_{b_{-}\rightarrow b_{+}}\frac{1+\frac{q\Lambda}{b_{+}-b_{-}}\int_{b_{-}}^{b_{+}}W^{(q)}(z)d z }{\frac{W^{(q)}(b_{+})-W^{(q)}(b_{-})}{b_{+}-b_{-}}}\\
&=\frac{1+q\Lambda W^{(q)}(b_{+})}{W^{(q)'}(b_{+})}=\zeta_{\Lambda}(b_{+}).
\end{align*}
\end{remark}

Now, from Proposition \ref{LagrangianbarrierTrans} we note that a candidate for optimal single band strategy would be a maximizer of the function $G_{\Lambda}$. The candidate to optimal levels $b^{\Lambda}=(b_-^{\Lambda},b_+^{\Lambda})$ are defined as follows:
\begin{equation}\label{bLambdaTrans}
\begin{cases}
 b^{\Lambda}_{-}=b^{*}(d^{*}),\\
 b^{\Lambda}_{+}= b^{\Lambda}_{-}+d^{*},
\end{cases}
\end{equation}
where
\begin{equation}
 \begin{cases}
  b^{*}(d):=\sup\{\eta\geq0:G_{\Lambda}(\eta,\eta+d)\geq G_{\Lambda}(\varsigma,\varsigma+d),\forall \varsigma\geq0 \},\ \text{with}\ d>0,\\
  d^{*}:=\sup\{d\geq0:G_{\Lambda}(b^{*}(d),b^{*}(d)+d)\geq G_{\Lambda}(b^{*}(\varsigma),b^{*}(\varsigma)+\varsigma),\forall \varsigma\geq0 \}.
 \end{cases}
\end{equation}
It can be verified that
\begin{equation}
 G_{\Lambda}(b^{\Lambda}_{-},b^{\Lambda}_{+})\geq G_{\Lambda}(b_{-},b_{+}),\ \text{for any}\ (b_{-},b_{+})\ \text{with}\ 0\leq b_{-}<b_{+}, 
\end{equation}
and from \cite{avram2015}, we get the following statement.
\begin{thm}[Optimal strategy for \eqref{P2}]\label{L2}
 Let $b^{\Lambda}=(b^{\Lambda}_{-},b^{\Lambda}_{+})$ be defined as in \eqref{bLambdaTrans}. Then, $b^{\Lambda}_{+}<\infty$ and 
 \begin{equation}\label{p1}
G_{\Lambda}(b^{\Lambda}_{-},b^{\Lambda}_{+})=\zeta_{\Lambda}(b^{\Lambda}_{+}),
\end{equation}
where $\zeta_{\Lambda}$ is given by \eqref{functionZ}. In particular, it is optimal to adopt the strategy $D^{b^{\Lambda}}$.
\end{thm}

\subsubsection{Solution of \eqref{P1}} 

In this section we take a slightly different approach than in the previous case. In this case we consider the parametric curve given by $\Lambda\mapsto b^{\Lambda}=(b_-^{\Lambda},b_+^{\Lambda})$ for $\Lambda\geq0$. The following lemma gives the relationship between $b_{\Lambda}$ and the optimal pair $(b_{-}^{\Lambda},b_{+}^{\Lambda})$, where $b_{\Lambda}$ is given by \eqref{bLambda}.

\begin{lemma}\label{bLambdaypair}
 Let $(b^{\Lambda}_{-},b^{\Lambda}_{+})$ be defined as in \eqref{bLambdaTrans}, where $\Lambda\geq0$. Then:
 \begin{enumerate}[(1)]
  \item If $b_{\Lambda}>0$, then $0\leq b^{\Lambda}_{-}< b_{\Lambda}<b^{\Lambda}_{+}$.
  \item If $b_{\Lambda}=0$, then $b_{-}^{\Lambda}=b_{\Lambda}<b_{+}^{\Lambda}$.
 \end{enumerate}
\end{lemma}
\begin{proof}
First note that under the assumption of completely monotonicity of the density of $\nu$, we have that $G_{\Lambda}$ is smooth, therefore we can compute its stationary points. So,  if $(b^{\Lambda}_{-},b^{\Lambda}_{+})$ is an interior maximum  point,  i.e. $0<b_{-}^{\Lambda}<b_{+}^{\Lambda}$, we must have that   
\begin{equation}\label{p13}
\nabla G_{\Lambda}(b_-,b_+)=\begin{pmatrix}
    \dfrac{W^{(q)'}(b_{-})}{W^{(q)}(b_{+})-W^{(q)}(b_{-})}\left(G_{\Lambda}(b_-,b_+)-\zeta_{\Lambda}(b_-)\right)\\
    -\dfrac{W^{(q)'}(b_{+})}{W^{(q)}(b_{+})-W^{(q)}(b_{-})}\left(G_{\Lambda}(b_-,b_+)-\zeta_{\Lambda}(b_+)\right)
  \end{pmatrix}=\begin{pmatrix}0\\0\end{pmatrix}
\end{equation}
Suppose now that  $b_{\Lambda}>0$. If $b_-^{\Lambda}$ is strictly positive, by \eqref{p13} it follows that 
\begin{equation}\label{p15}
\zeta_{\Lambda}(b^{\Lambda}_{-})=G_{\Lambda}(b^{\Lambda}_{-},b^{\Lambda}_{+})=\zeta_{\Lambda}(b^{\Lambda}_{+}),
\end{equation}
and by Remark \ref{complmonotprop}, this means that $b^{\Lambda}_{-}< b_{\Lambda}<b^{\Lambda}_{+}$. If $b_{-}^{\Lambda}=0$, define the function $g_{\Lambda}:(0,\infty)\longrightarrow\R$ as $g_{\Lambda}(\varsigma):=G_{\Lambda}(0,\varsigma)$. Then,
\begin{equation*}
 g'_{\Lambda}(\varsigma)=\frac{W^{(q)'}(\varsigma)}{W^{(q)}(\varsigma)-W^{(q)}(0)}[\zeta_{\Lambda}(\varsigma)-g_{\Lambda}(\varsigma)],
\end{equation*}
and since $b_{+}^{\Lambda}$ is a maximizer of $g_{\Lambda}$, it follows that $\zeta_{\Lambda}(b_{+}^{\Lambda})=g_{\Lambda}(b_{+}^{\Lambda})$ and
\begin{equation*}
 \begin{cases}
  \zeta_{\Lambda}(\varsigma)>g_{\Lambda}(\varsigma),& \text{if}\ \varsigma< b_{+}^{\Lambda}\\
  \zeta_{\Lambda}(\varsigma)<g_{\Lambda}(\varsigma),& \text{if}\ \varsigma> b_{+}^{\Lambda},
 \end{cases}
\end{equation*}
which shows that $b_{+}^{\Lambda}>b_{\Lambda}$, again by Remark \ref{complmonotprop}. 
In the case where $b_{\Lambda}=0$, we must have that $b_{-}^{\Lambda}=0$. Otherwise, if $b_{-}^{\Lambda}\neq0$, from \eqref{p13}, we have that $\zeta_{\Lambda}(b_{-}^{\Lambda})=\zeta_{\Lambda}(b_{+}^{\Lambda})$, which is a contradiction since $\zeta_{\Lambda}$ is a strictly decreasing function on $(0,\infty)$.
\end{proof}

\begin{prop}\label{lambdainftylemmaTrans}
The curve $\Lambda\mapsto(b_-^{\Lambda},b_+^{\Lambda})$ for $\Lambda\geq0$ is continuous and unbounded.
\end{prop}
\begin{proof}
The previous lemma and Corolary \ref{bLambdainfty} shows that $b_{+}^{\Lambda}\rightarrow\infty$ as $\Lambda\rightarrow\infty$, so the curve is unbounded. The continuity follows from the Implicit Function Theorem by considering two cases. First, suppose $b_-^{\Lambda}=0$, then by Theorem \ref{L2} we can define $b_+^{\Lambda}$ by the equation 
$$F(\Lambda,b_+^{\Lambda}):=G_{\Lambda}(0,b_+^{\Lambda})-\zeta_{\Lambda}(b_+^{\Lambda})=0.$$
Simple calculations show that 
$$\frac{\partial F}{\partial b_+}(\Lambda,b_+^{\Lambda})=\frac{\partial G_{\Lambda}}{\partial b_+}(0,b_+^{\Lambda})-\zeta'_{\Lambda}(b_+^{\Lambda})=-\zeta'_{\Lambda}(b_+^{\Lambda})>0,$$
since $b_+^{\Lambda}>b_{\Lambda}$, so the conditions of the Implicit Function Theorem are satisfied. Now, if $b_-^{\Lambda}>0$, the optimal pair is defined by the equations $F(\Lambda,b_-^{\Lambda},b_+^{\Lambda})=(F_1(\Lambda,b_-^{\Lambda},b_+^{\Lambda}),F_2(\Lambda,b_-^{\Lambda},b_+^{\Lambda}))=(0,0)$, where
\begin{align*}
F_1(\Lambda,b_-^{\Lambda},b_+^{\Lambda}):=G_{\Lambda}(b_-^{\Lambda},b_+^{\Lambda})-\zeta_{\Lambda}(b_-^{\Lambda})&=0,\\
F_2(\Lambda,b_-^{\Lambda},b_+^{\Lambda}):=G_{\Lambda}(b_-^{\Lambda},b_+^{\Lambda})-\zeta_{\Lambda}(b_+^{\Lambda})&=0.
\end{align*}
Again, simple calculations show that the Jacobian determinant of this system of equations is $\zeta'_{\Lambda}(b_+^{\Lambda})\zeta'_{\Lambda}(b_-^{\Lambda})<0$, since $b^{\Lambda}_{-}< b_{\Lambda}<b^{\Lambda}_{+}$, implying the continuity of the curve.
\end{proof}

Next, we proceed to analyze the level curves of the constraint. From Equations \eqref{ValuefunctqscaleTrans} and \eqref{lagrangianbarrierTrans} we observe that for $b=(b_-,b_+)$
\begin{align}\label{PsiTrans}\nonumber
\varPsi_x(b_{-},b_{+}):&=\E_x\left[e^{-q\tau^{D^b}}\right]\\
&=
\begin{cases} 
Z^{(q)}(x)-W^{(q)}(x)\dfrac{q\int_{b_{-}}^{b_{+}}W^{(q)}(z)d z}{W^{(q)}(b_{+})-W^{(q)}(b_{-})}, & \mbox{if } 0\leq x\leq b_{+}\\
\dfrac{Z^{(q)}(b_{-})W^{(q)}(b_{+})-Z^{(q)}(b_{+})W^{(q)}(b_{-})}{W^{(q)}(b_{+})-W^{(q)}(b_{-})},  &   \mbox{if } x>b_{+}.
\end{cases}
\end{align}
\begin{remark}
Note that
\begin{equation*}
\lim_{b_{-}\rightarrow b_{+}}\varPsi_{x}(b_{-},b_{+})=\varPsi_x(b_+),
\end{equation*}
where $\varPsi_x(b)$ is defined in \eqref{Psi}.
\end{remark}

The next few lemmas will describe the properties of the level curves of the function \eqref{PsiTrans}.
\begin{lemma}\label{Psidecr}
Let $x\geq0$ be fixed. 
\begin{enumerate}[(i)]
 \item  If $b_-\geq0$ is fixed, the function $\varPsi_x(b_{-},b_{+})$, given in \eqref{PsiTrans}, is non-increasing  for all $b_{+}>b_{-}$, and
 \begin{equation}\label{KlimitTrans}
\lim_{b_{+}\rightarrow\infty}\varPsi_x(b_{-},b_{+})=\bar{K}_{x},
 \end{equation}
where $\bar{K}_x$ is defined in \eqref{Klimitclass}.
 \item  If $b_{+}>0$ is fixed, $\varPsi_x(b_{-},b_{+})$ is non-increasing  for all $b_{-}\in[0,b_{+})$.
\end{enumerate}
\end{lemma}
\begin{proof}
First, assume that $x\leq b_+$. To show that $\varPsi_x(b_{-},b_{+})$ is non-increasing, it is sufficient to verify that 
\begin{equation}\label{p5.0}
\frac{\int_{b_{-}}^{b_{+}}W^{(q)}(z) d z}{W^{(q)}(b_{+})-W^{(q)}(b_{-})},
\end{equation}
is non-decreasing, which is true if 
\begin{align}\nonumber
\frac{\partial}{\partial b_{+}}\biggr[&\frac{\int_{b_{-}}^{b_{+}}W^{(q)}(z)d z}{W^{(q)}(b_{+})-W^{(q)}(b_{-})}\biggl]\\\label{p5}
&=\frac{W^{(q)}(b_{+})}{W^{(q)}(b_{+})-W^{(q)}(b_{-})}-\frac{W^{(q)'}(b_{+})\int_{b_{-}}^{b_{+}}W^{(q)}(z) d z}{[W^{(q)}(b_{+})-W^{(q)}(b_{-})]^{2}}\geq0.
\end{align}
Since $W^{(q)}$ is a log-concave function on $[0,\infty)$, we have that 
\begin{equation}\label{p4}
\frac{W^{(q)'}(\eta)}{W^{(q)}(\eta)}\geq\frac{W^{(q)'}(\varsigma)}{W^{(q)}(\varsigma)},\ \text{for any}\ \eta \ \text{and}\ \varsigma\ \text{with}\ \eta\leq\varsigma. 
\end{equation}
Taking $\varsigma=b_{+}$ in the above inequality, it follows that 
\begin{equation*}
W^{(q)'}(\eta)\geq\frac{W^{(q)'}(b_{+})}{W^{(q)}(b_{+})}W^{(q)}(\eta),\ \text{for any}\ \eta\in[b_{-},b_{+}]. 
\end{equation*}
Then, integrating between $b_{-}$ and $b_{+}$, it yields \eqref{p5} and hence \eqref{p5.0} is non-decreasing. For the case $x>b_+$, if $b_-=0$ and $W^{(q)}(0)=0$ we obtain the constant 1. Otherwise, similar calculations as above show that the function is non-increasing. Proceeding in a similar way that before, we also obtain (ii). Now, by \eqref{Klimitclass} and L'H\^opital's rule, it is easy to see \eqref{KlimitTrans} for any of $b_-$.
\end{proof}
\begin{remark}
Note that if $b_+>b_0$ we obtain strictly decreasing functions in the above lemma by Corollary \ref{strictlogconcave}.
\end{remark}

\begin{lemma}\label{Kcurve}
Let $x\geq0$. Then, for each $K\in(\bar{K}_{x},\varPsi_x(0))$ there exist $\underline{b}$ and $\bar{b}$ such that the level curve $L_K(\varPsi_x)=\{(b_-,b_+):\varPsi_x(b_-,b_+)=K\}$ is continuous and contained in $[0,\underline{b}]\times[\underline{b},\bar{b}]$.
\end{lemma}
\begin{proof}
The continuity of the level curve is an immediate consequence of the continuity of $\varPsi_x(\cdot,\cdot)$. First, observe that by Proposition \ref{optimalpair} we know the existence of $\underline{b}\geq0$ such that $\varPsi_x(\underline{b})=K$ (if it is not unique, we take the minimum of them).  On the other hand, by Lemma \ref{Psidecr} there exists $\bar{b}\in[\underline{b},\infty)$ such that $\varPsi_x(0,\bar{b})=K$ (again, if it is not unique, we take the maximum of them). Now, the fact that the curve is contained in $[0,\underline{b}]\times[\underline{b},\bar{b}]$ is again consequence of Lemma \ref{Psidecr}.
\end{proof}

We now prove the result analogous to Propositions \ref{optimalpair}.

\begin{prop}\label{optLambda}
Let $x\geq0$. Then, for each $K>\bar{K}_x$ there exists $\Lambda^*\geq0$ such that:
\begin{enumerate}[(i)]
\item $\varPsi_x(b_-^{\Lambda^*},b_+^{\Lambda^*})=\mathbb{E}_{x}\left[e^{-q \tau^{D^{b_{\Lambda^*}}}}\right]\leq K$ and
\item\ $\Lambda^*\left(K-\mathbb{E}_{x}\left[e^{-q \tau^{D^{b_{\Lambda^*}}}}\right]\right)=0$.
\end{enumerate}
\end{prop}
\begin{proof}
If $K\geq\varPsi_x(b_-^0,b_+^0)$, the the unconstrained problem satisfies the restriction and $\Lambda^*=0$ satisfies the conditions. Otherwise, by Proposition \ref{lambdainftylemmaTrans} and Lemma \ref{Kcurve} we deduce that the parametric curve $\Lambda\mapsto b_{\Lambda}=(b_-^{\Lambda},b_+^{\Lambda})$ and the level curve $L_K(\varPsi_x)$ must intersect, that is, there exist $\Lambda^*$ such that $\varPsi_x(b_-^{\Lambda^*},b_+^{\Lambda^*})=K$, so satisfies the conditions.
\end{proof}

By similar arguments as in the previous case we can show the absence of duality gap also in this case.

\section{Solution of the constrained Dual Model}\label{SecDual}

Let us consider the Dual model where the reserve process $X$ is a spectrally positive L\'evy process. In this model we will only study the case without transaction cost. The other case should be a straightforward applications of the ideas presented in this article. 

In order to consider the barrier strategy at level $b$, we need to construct the reflected process at its supremum with initial value $b$, as before. To do so, we note that
\begin{equation}\label{posnegrel}
\hat{X}^{b}=(b\vee \overline{X})-X=Y-(0\wedge\underline{Y}),
\end{equation}
where $Y=b-X$ and $\underline{Y}_t:=\underset{0\leq s\leq t}{\inf} Y_s$. Note that $Y$ is now a spectrally negative L\'evy processes, hence the useful identities in this case concern the  \emph{reflected process at its infimum}.  Therefore when we refer to the Dual model, $q$-scale functions and other quantities correspond to the process $-X$. 

We now present the equivalent version of Proposition \ref{Valuefunctqscale} for the Dual model. The proof of the following proposition is available in \cite{AvPaPi07,KyYa14}. 

\begin{prop}
Let $b>0$ and consider the dividend process $D_t^{b}=X_t -(b- \hat{X}^b_t)$ and $X$ a spectrally positive L\'evy process. For $x\in [0,b]$,
\begin{equation}\label{ValuefunctqscaleDual}
\V^{D^b}(x)=\E_x\left[\int_0^{\tau^{D^b}-} e^{-qt}dD_t^b\right]=-k(b-x)+\frac{Z^{(q)}(b-x)}{Z^{(q)}(b)}k(b),
\end{equation}
where $k(\varsigma):=\bar{Z}^{(q)}(\varsigma)-\dfrac{1}{\Phi(q)}Z^{(q)}(\varsigma)+\dfrac{\psi'(0+)}{q},\, \varsigma\geq 0$.
\end{prop}

\subsection{Solution of \eqref{P2}}

We also need the equivalent result for the function $\V_\Lambda^{D^b}$.

\begin{prop}
The function $\V_\Lambda^{D^b}$, where $D^b$ is the barrier strategy at level $b \geq 0$, for $x\geq0$ is given by
\begin{align}\label{lagrangianbarrierDual}
\V_\Lambda^{D^b}(x)=
\begin{cases}
-k(b-x)+\frac{Z^{(q)}(b-x)}{Z^{(q)}(b)}\left[k(b)-\Lambda\right] + \Lambda K, &\text{if}\quad x\leq b\\
x-b+\mathcal{V}_\Lambda^{D^b}(b), &\text{if}\quad x>b. 
\end{cases}
\end{align}
\end{prop}
\begin{proof}
First note that $Z^{(q)}(z)=1$ and $\bar{Z}^{(q)}(z)=z$ for $z<0$. Now, from \cite{AvPaPi07} we have that if $Y$ is a spectrally negative L\'evy process and $\tilde{Y}=Y-(0\wedge\underline{Y})$, the reflected process at its past infimum below 0, then $\E_{y}\left[e^{-q \tau_b}\right]=\frac{Z^{(q)}(y)}{Z^{(q)}(b)}$, where $\tau_b$ is the first hitting time of $\tilde{Y}$ at $\{b\}$. Therefore, for $X$ a spectrally positive L\'evy process, by \eqref{posnegrel}, one gets
\begin{equation}\label{eqrestdual}
\E_{x}\left[e^{-q \tau^{D^b}}\right]=\frac{Z^{(q)}(b-x)}{Z^{(q)}(b)}.
\end{equation}
Combining this with the previous proposition yields the result.
\end{proof}
The last result is also included in \cite{Yin}. Now, to solve \eqref{P2} in the set up of the Dual model we will follow \cite{KyYa14} closely. Again, the idea is to propose a candidate for optimal barrier and run it through a verification lemma. In contrast to the approach taken in subsection \ref{dualclassical}, the candidate barrier will be such that corresponding value function is $C^1$ [resp. $C^2$] in the case of bounded [resp. unbounded] variation. This approach is commonly referred as \emph{smooth fit}. In this section no assumption about the L\'evy measure $\nu$ is made.

From Equation \eqref{lagrangianbarrierDual} we get,
\begin{align*}
(\V_\Lambda^{D^b})'(x)=Z^{(q)}(b-x)-q W^{(q)}(b-x)\xi_\Lambda(b),
\end{align*}
and
\begin{align*}
(\V_\Lambda^{D^b})''(x)=-q W^{(q)}(b-x)+qW^{(q)'}(b-x)\xi_\Lambda(b),
\end{align*}
where
\begin{align}
\xi_\Lambda(\varsigma):=\frac{1}{\Phi(q)}+\frac{k(\varsigma)-\Lambda}{Z^{(q)}(\varsigma)}.
\end{align}

It is easy to check that the smooth fit condition is equivalent to $\xi_\Lambda(b)=0$ in both the bounded or unbounded variation case. This is equivalent to $b$ satisfying the relation $\bar{Z}^{(q)}(b)=\Lambda-\frac{\psi'(0+)}q$. Finally, since $\bar{Z}^{(q)}$ is strictly increasing and $\bar{Z}^{(q)}(0)=0$ the candidate to optimal barrier is given by
\begin{align}\label{optbarrierdual}
b_\Lambda:= \begin{cases}
(\bar{Z}^{(q)})^{-1}\left(\Lambda-\frac{\psi'(0+)}{q}\right)& \mbox{if } \frac{\psi'(0+)}{q}<\Lambda\\
0 & \mbox{otherwhise.}
\end{cases}
\end{align}
This level is indeed optimal. To see that, we can use the standard verification lemma approach as in Proposition 5 in \cite{AvPaPi07} and Theorem 2.1 in \cite{KyYa14}. This is also shown in \cite{Yin}.

\begin{thm}[Optimal strategy for \eqref{P2}]
The optimal strategy of \eqref{P2} consist of a barrier strategy at level $b_\Lambda$ given by \eqref{optbarrierdual}, and the corresponding value function is given by \eqref{lagrangianbarrierDual}.
\end{thm}

\subsection{Solution of \eqref{P1}}
As in the previous section let $b_0$ be the optimal barrier for \eqref{P2} with $\Lambda=0$, that is, the optimal barrier for the unconstrained problem, and let  $\bar{\Lambda}:=\sup \{\Lambda\geq 0: b_\Lambda=0\}\vee0$. We also consider the function $\Lambda:[b_0,\infty)\rightarrow \R_{+}$ defined by 
$$\Lambda(b):= \begin{cases}
0& \mbox{if } b=b_0\\
\bar{Z}^{(q)}(b)+\frac{\psi'(0+)}{q} &  \mbox{if } b>b_0.
\end{cases}
$$

\begin{prop}
For each $b \in (b_0,\infty)$ the barrier strategy at level $b$ is optimal for \eqref{P2} with $\Lambda(b)$. Also, this map is one-to-one onto $(\bar{\Lambda},\infty)$.
\begin{proof}
For $b\in (b_0,\infty)$ the . First, it is strictly increasing and goes to $\infty$ as $b$ goes to $\infty$, since so satisfies $\bar{Z}^{(q)}$. Finally,  $\Lambda(b)\geq0$ and satisfies the optimality condition by \eqref{optbarrierdual}.
\end{proof}
\end{prop}

Now we show that the complementary slackness condition if satisfied.

\begin{prop}\label{optimalpairdual}
For each $x\geq0$ there exists $\bar{K}_{x}\geq0$ such that if $K>\bar{K}_{x}$ there exists $b^*$ such that :
\begin{enumerate}[(i)]
\item $\mathbb{E}_{x}\left[e^{-q \tau^{D^{b^*}}}\right]\leq K$ and
\item $\Lambda(b^*)\left(K-\mathbb{E}_{x}\left[e^{-q \tau^{D^{b^*}}}\right]\right)=0$.
\end{enumerate}
\begin{proof}
Again, let $\varPsi_x(b):=\mathbb{E}_{x}\left[e^{-q \tau^{D^b}}\right]$ given by 
$$\varPsi_x(b)=\frac{Z^{(q)}(b-x)}{Z^{(q)}(b)}.$$
Note that this expression is valid for any $b\geq0$. Let $x>0$. To get that $\frac{d \varPsi_x(b)}{db}< 0$, a simple calculation shows that this condition is equivalent to $q\frac{W^{(q)}(b-x)}{Z^{(q)}(b-x)}<q\frac{W^{(q)}(b)}{Z^{(q)}(b)}$. This is true since $\ln (Z^{(q)}(x))$ is strictly increasing. Now, using \eqref{limitqfact}, we define $\bar{K}_{x}:=\lim\limits_{b\rightarrow\infty}\varPsi_x(b)=e^{-\Phi(q) x}$. Now, the proof follows identically as in Proposition \ref{optimalpair}. For the case $x=0$, note that $\bar{K}_{0}=\varPsi_0(b)=1$ for all $b$, so $b^*=b_0$ satisfies the conditions.
\end{proof}
\end{prop}

A remark analogous to \ref{remdonothing} also holds in this model and therefore we can also prove the next lemma.
\begin{lemma}\label{limit}
Let $x\geq0$. If $K=\bar{K}_{x}$ then $\Lambda(b)\Big(K-\mathbb{E}_x\Big[e^{-q \tau^{D^b}}\Big]\Big)\rightarrow 0$ as $b\rightarrow \infty$.
\end{lemma}

As in the previous section we derive the main result.

\begin{thm}\label{strongduality} Let $x\geq 0$, $K\geq0$ and $V(x)$ be the optimal solution to \eqref{P1}. Then $$V(x)\geq \underset{\Lambda\geq 0}\inf\,\,V_{\Lambda}(x)$$
 and therefore, $\underset{\Lambda\geq 0}\inf\,\,V_{\Lambda}(x)=V(x)$.
\end{thm}

\section{Numerical example}\label{numerics}

In this section we illustrate with numerical examples the previous results. The main difficulty here is that, in most cases, there are no closed form expression for scale functions. Hence, we will follow a numerical procedure presented in \cite{surya2008} to approximate the scale functions by Laplace transform inversion of \eqref{wqlaplace}. We do the same to approximate derivatives of the scale functions and use the trapezoidal rule to calculate integrals of it.

\begin{ex}\label{exdeFin}
In this example we consider the Cram\'er-Lundberg model with income premium rate $c=1$, Poisson process intensity $\lambda=1$ and a heavy-tailed Pareto Type II distributed claims with density function $p(x)=1.5\left(1+x\right)^{-2.5}$, also know as Lomax(1,1.5). Note that this density is a completely monotone function. In this example $q=0.05$. In this case the optimal barrier for the unconstrained problem, that is, $b_0=0.42$. Figure \ref{figrest} shows the function $\Psi_b(x)$ for different values of $b$. The figure also shows pairs $(x,K)$ for which the problem has biding, infeasible and inactive constraints and when the do-nothing strategy is optimal. The latter corresponds to pairs of the form $(x,\bar{K}_x)$.  A plot of the map $\Lambda(b)$ is presented in Figure \ref{figmap}. In this case we obtain that $\bar{\Lambda}=0$ since $b_0>0$, and a strictly increasing map on $[b_0,\infty)$.

\begin{figure}[h!]
   \begin{subfigure}[b]{.48\linewidth}
      \includegraphics[width=1\textwidth]{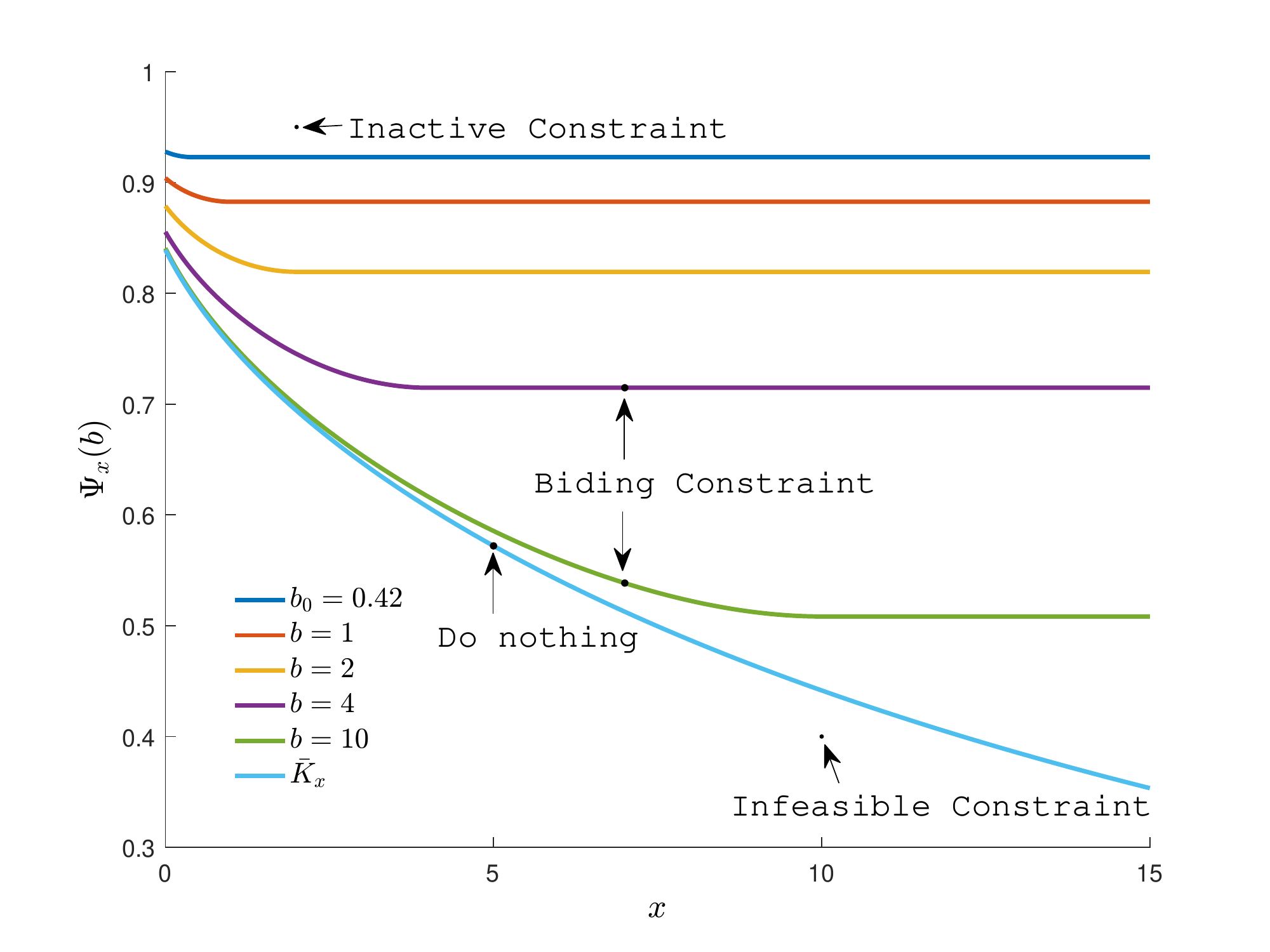}
	\caption{Example \ref{exdeFin}}
   \end{subfigure}
\begin{subfigure}[b]{.48\linewidth}
      \includegraphics[width=1\textwidth]{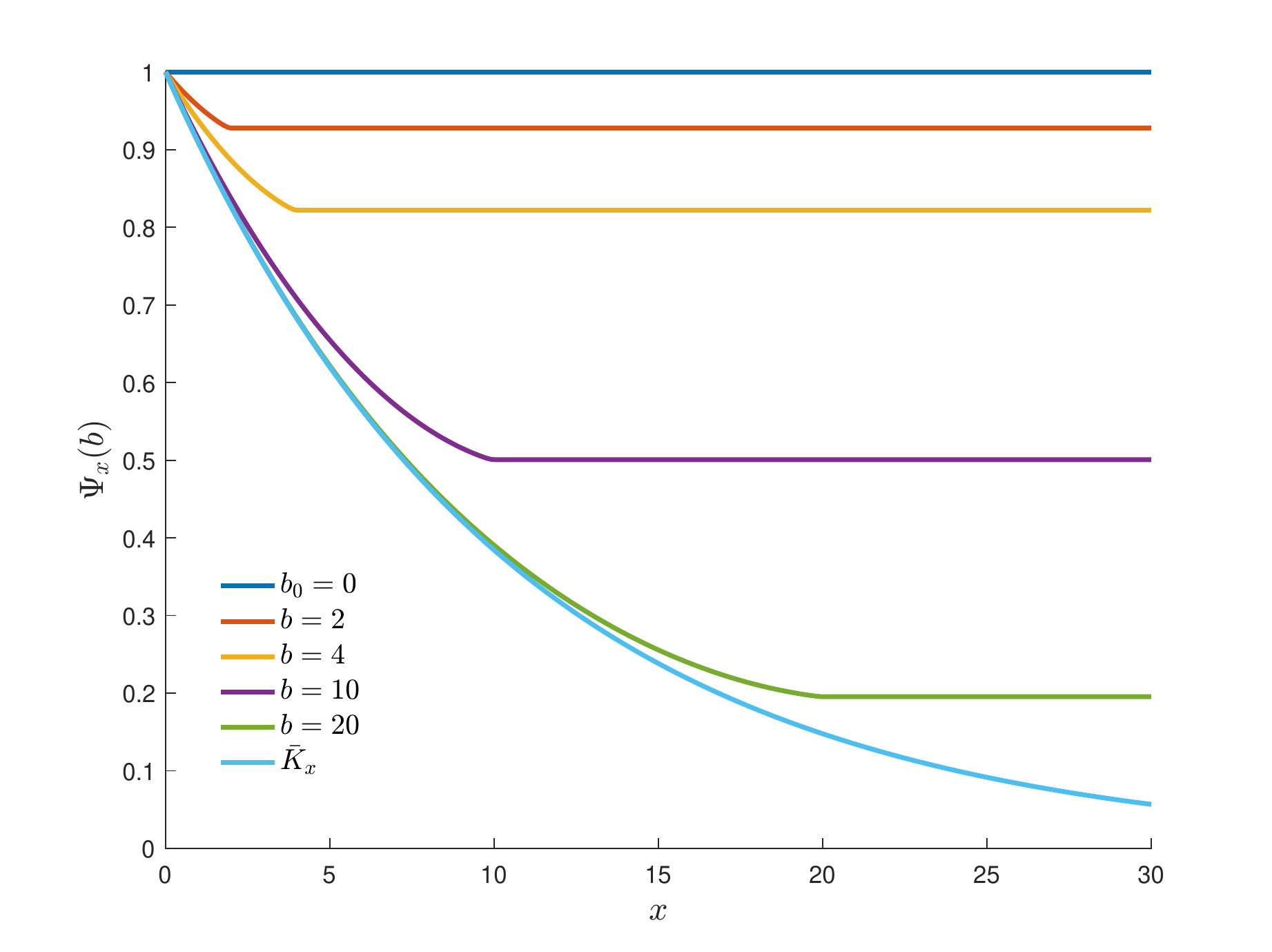}
	\caption{Example \ref{exDual}}
   \end{subfigure}
      \caption{$\Psi_x(b)$ as a function of $x$ for different values of $b$.}\label{figrest}
\end{figure}

\begin{figure}[t!]
   \begin{subfigure}[b]{.48\linewidth}
      \includegraphics[width=1\textwidth]{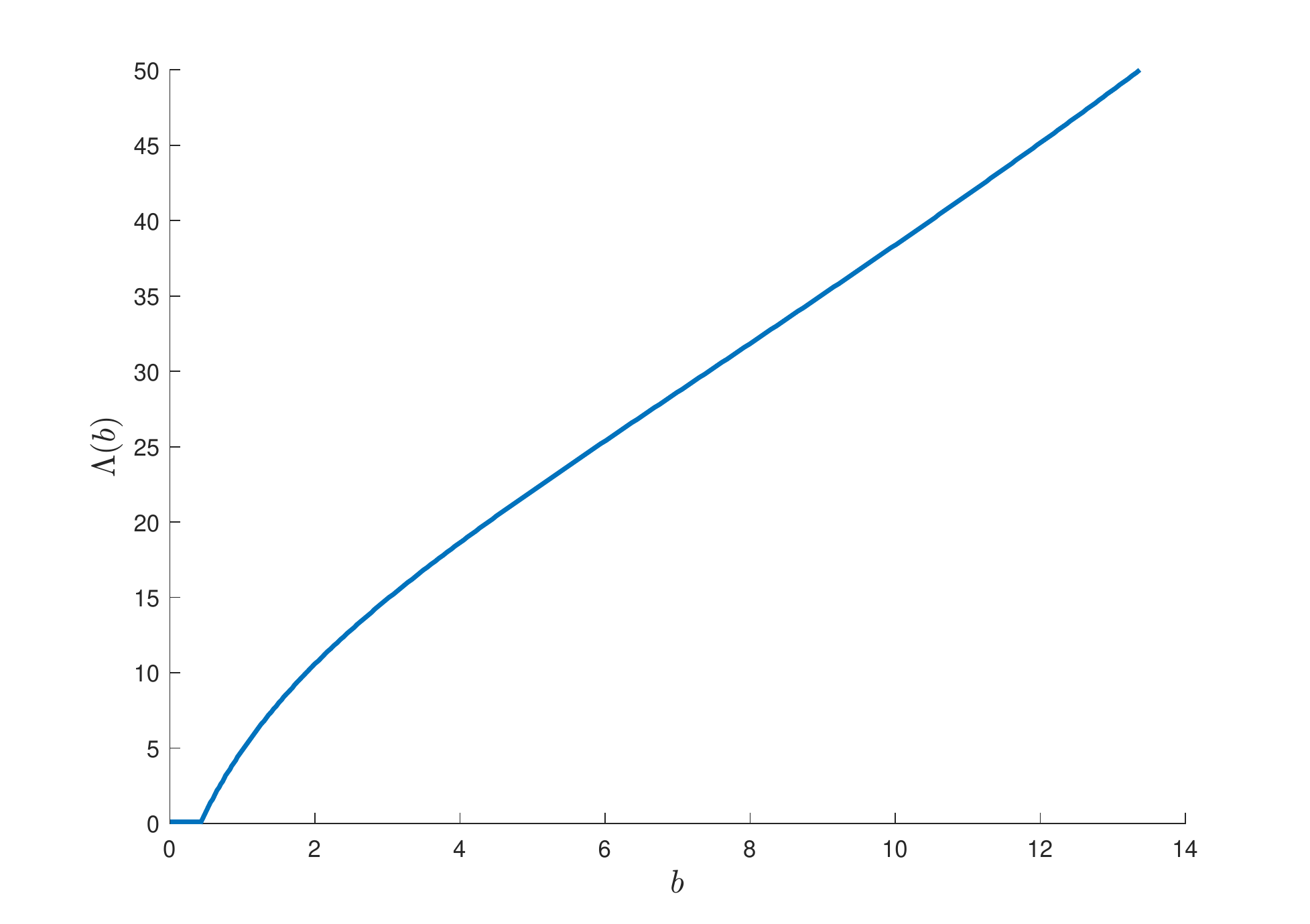}
	\caption{Example \ref{exdeFin}}
   \end{subfigure}
\begin{subfigure}[b]{.48\linewidth}
      \includegraphics[width=1\textwidth]{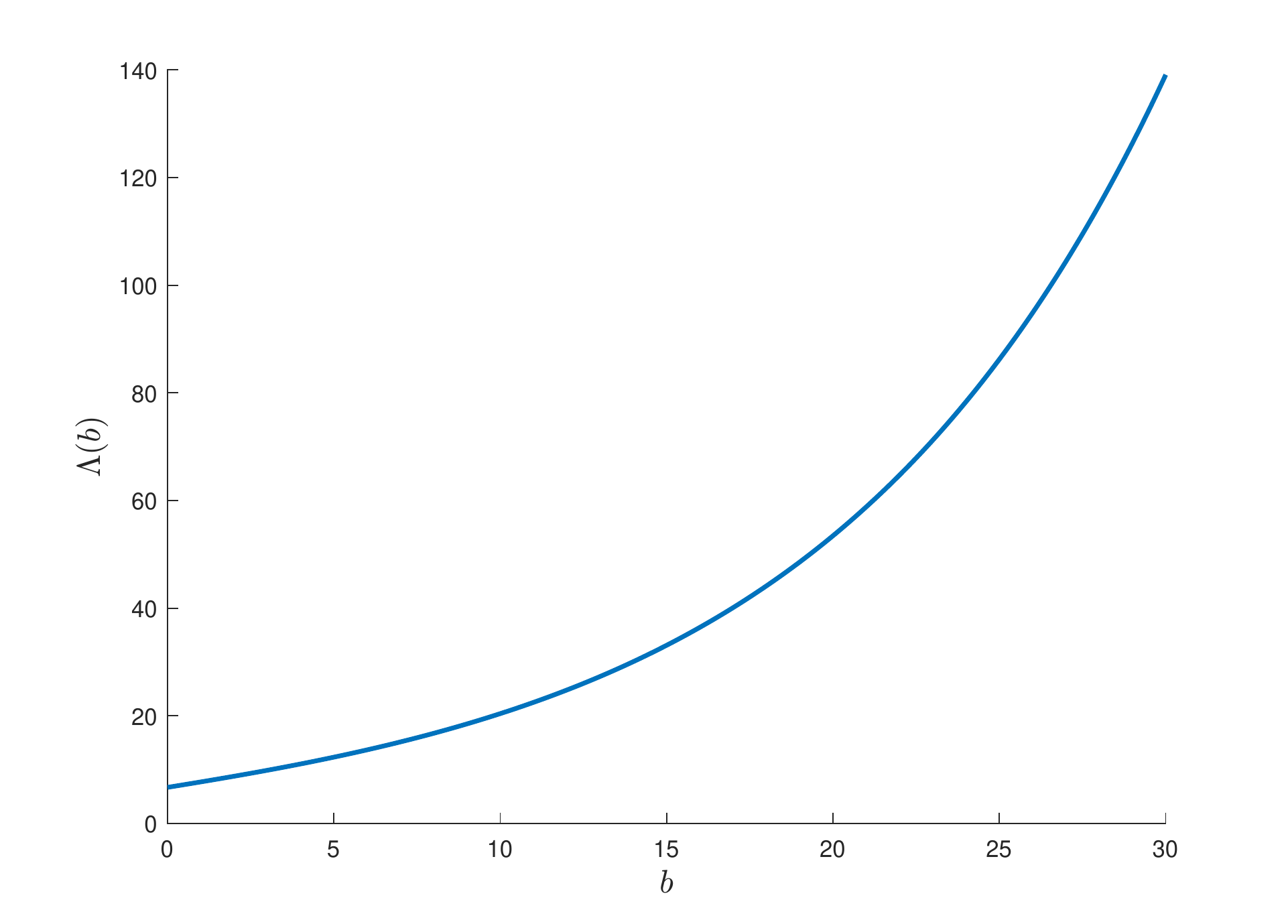}
	\caption{Example \ref{exDual}}
   \end{subfigure}
      \caption{Map $\Lambda(b)$.}\label{figmap}
\end{figure}

Figure \ref{figbs} shows the value of $b^*$ from Proposition \ref{optimalpair} for different values of $K$ as a function of $x$. From this figure we can extract the optimal policy for the constrained problem. Note that for some values of $x$ there is no $b^*$, for these values the problem is infeasible. The figure also shows the value of $b_0$ for reference. The value function for the unconstrained and constrained problems for few levels of $K$ is showed in Figure \ref{figoptimal}.
\end{ex}

\begin{figure}[t!]
   \begin{subfigure}[b]{.48\linewidth}
      \includegraphics[width=1\textwidth]{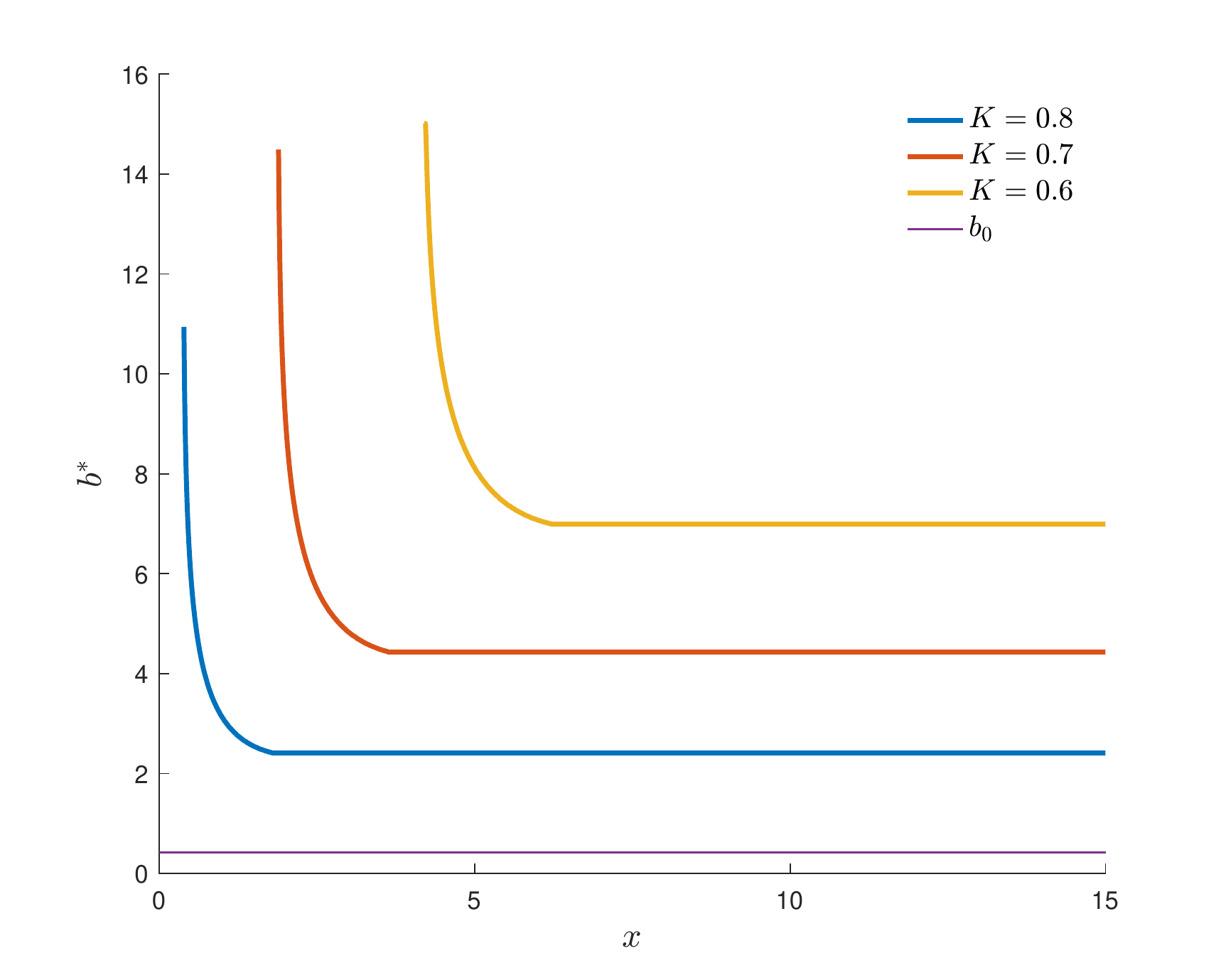}
	\caption{Example \ref{exdeFin}}
   \end{subfigure}
\begin{subfigure}[b]{.48\linewidth}
      \includegraphics[width=1\textwidth]{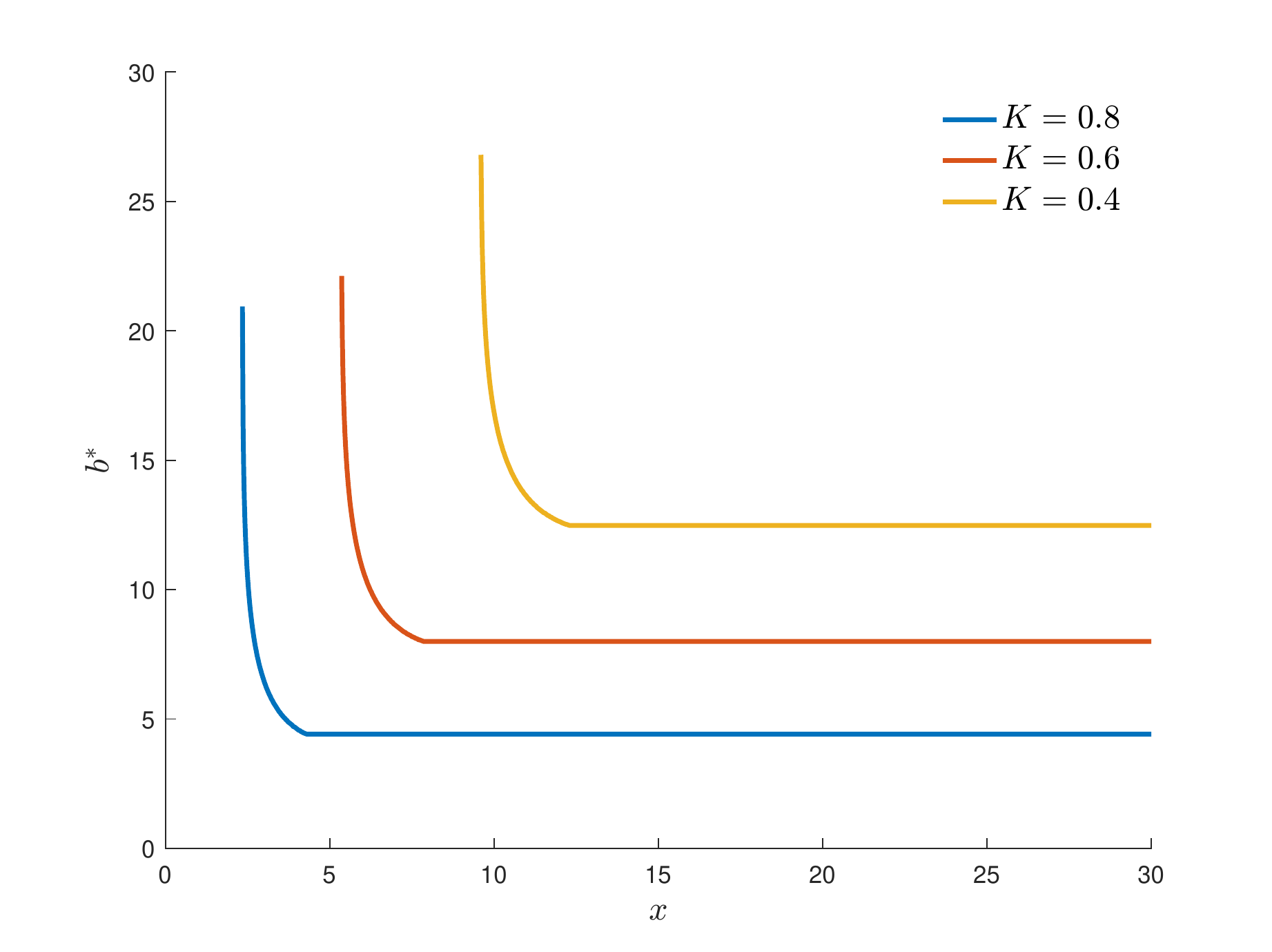}
	\caption{Example \ref{exDual}}
   \end{subfigure}
      \caption{$b^*$ for fixed values of $K$ as a function of $x$.}\label{figbs}
\end{figure}

\begin{figure}[t!]
   \begin{subfigure}[b]{.48\linewidth}
      \includegraphics[width=1\textwidth]{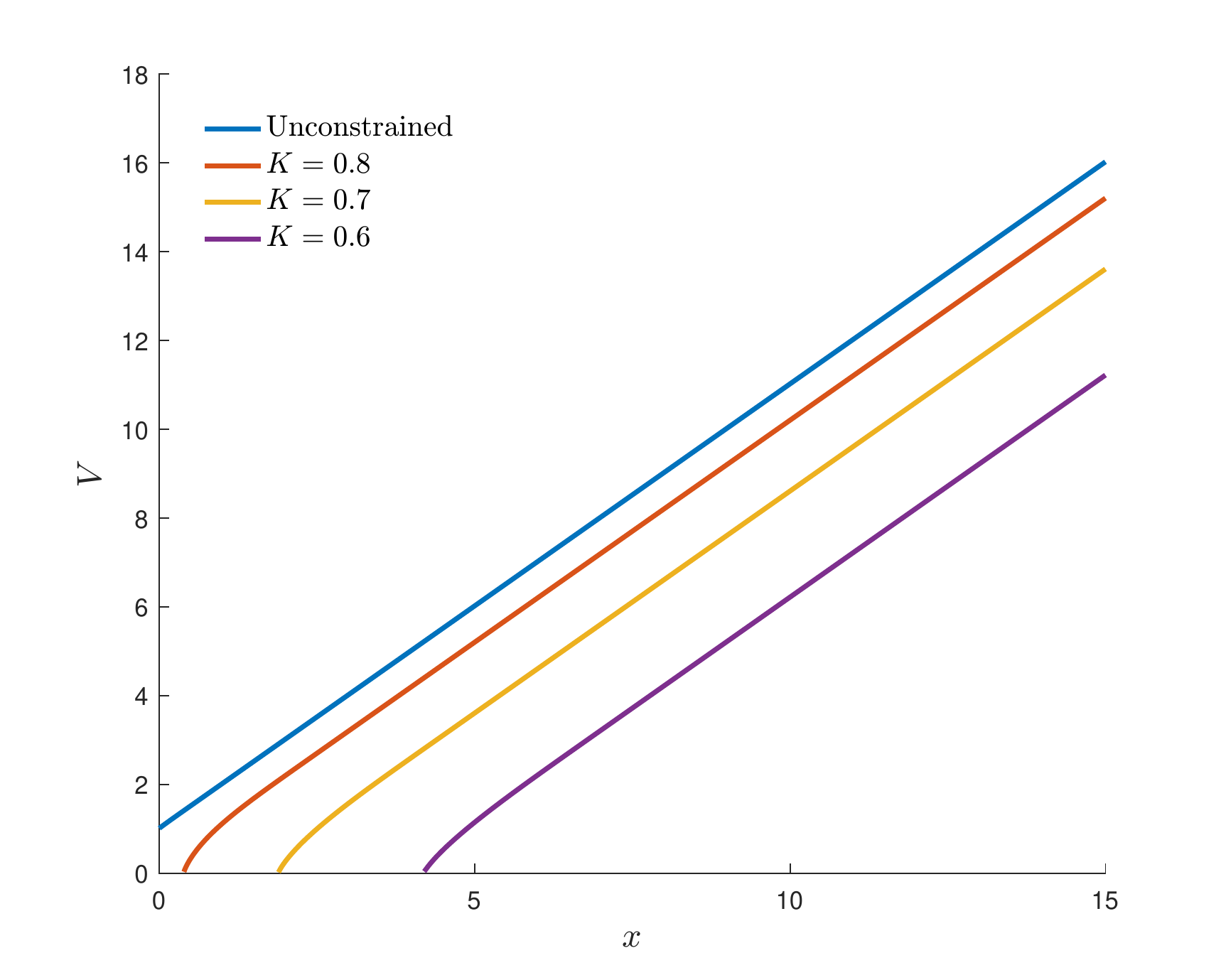}
	\caption{Example \ref{exdeFin}}
   \end{subfigure}
\begin{subfigure}[b]{.48\linewidth}
      \includegraphics[width=1\textwidth]{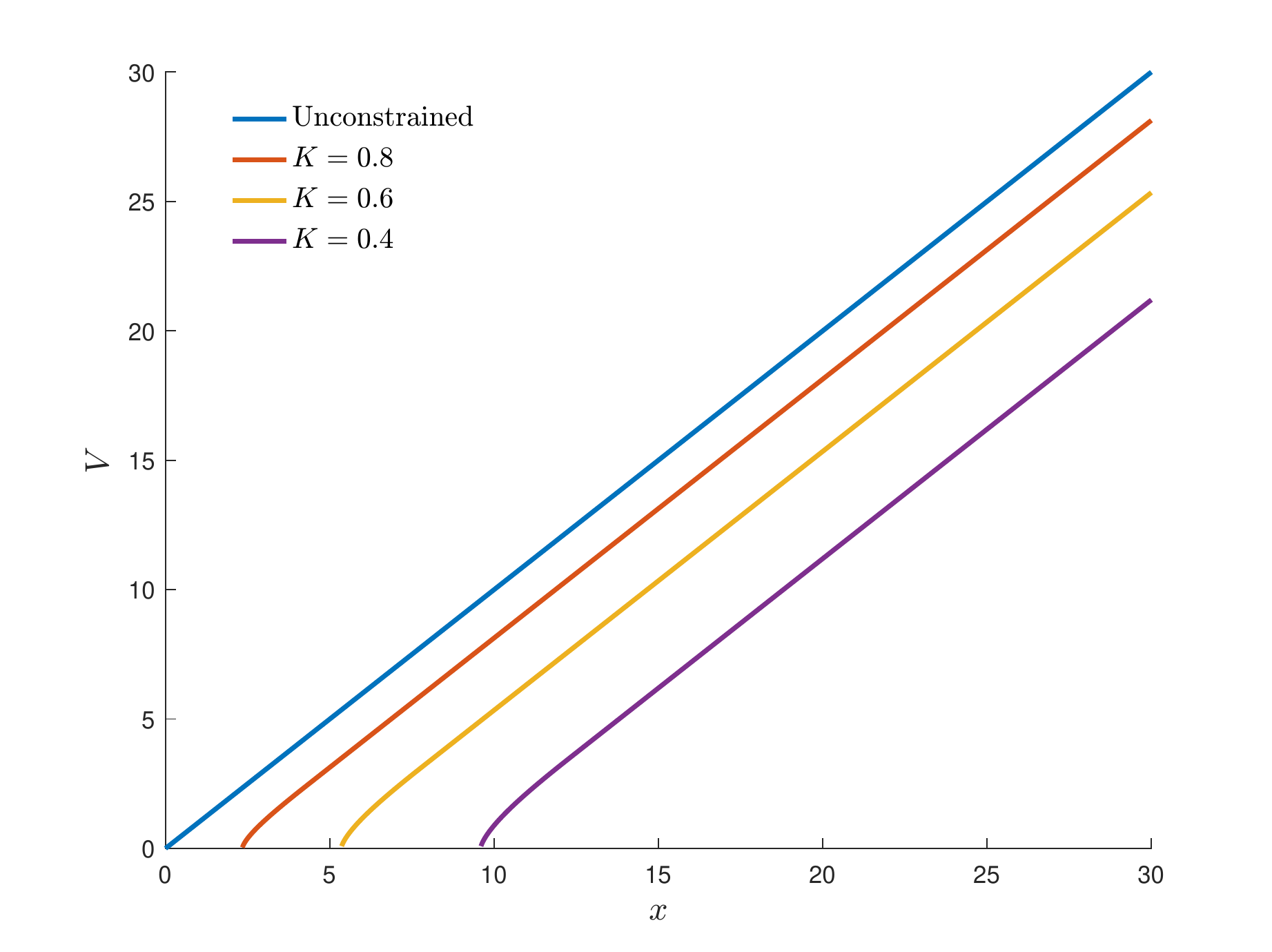}
	\caption{Example \ref{exDual}}
   \end{subfigure}
      \caption{Optimal value function $V$ for the constrained problem with different values of $K$.}\label{figoptimal}
\end{figure}

\begin{ex}\label{exDual}
We now consider problem with the Dual model. In this example the reserves process is $-X$, where $X$ follows the Cram\'er-Lundberg model plus a diffusion. The parameters are $c=1$, the intensity of the jumps is $\lambda=0.4$, the distribution of the jumps is Gamma(2,1) (note that this distribution doesn't have a completely monotone density) and $\sigma=0.5$. In this example $q=0.03$. In this case $b_0=0$ and $\bar{\Lambda}=6.71$. The results are shown in the same figures as in the previous example. Since $b_0=0$, if the problem is feasible then the constraint is active, Figure \ref{figrest}.
\end{ex}

\begin{ex}
We now consider an example of the problem wit transaction cost. The reserves process follows an $\alpha$-stable L\'evy process with $\alpha=1.5$. Note that this is a pure jump unbounded variation process. We take $q=0.1$ an the transaction cost $\beta=0.01$. Figure \ref{figPhiTran} shows contour plots of the function $\Psi_x(b_-,b_+)$ for $x=3,10$. This figure also shows the curve described by the map $\Lambda\mapsto(b_-^{\Lambda},b_+^{\Lambda})$. The values of the optimal pair as a function of $\Lambda$ and the values of $b_{\Lambda}$ are shown in Figure \ref{figoptb}.

\begin{figure}[t!]
   \begin{subfigure}[b]{.495\linewidth}
      \includegraphics[width=1\textwidth]{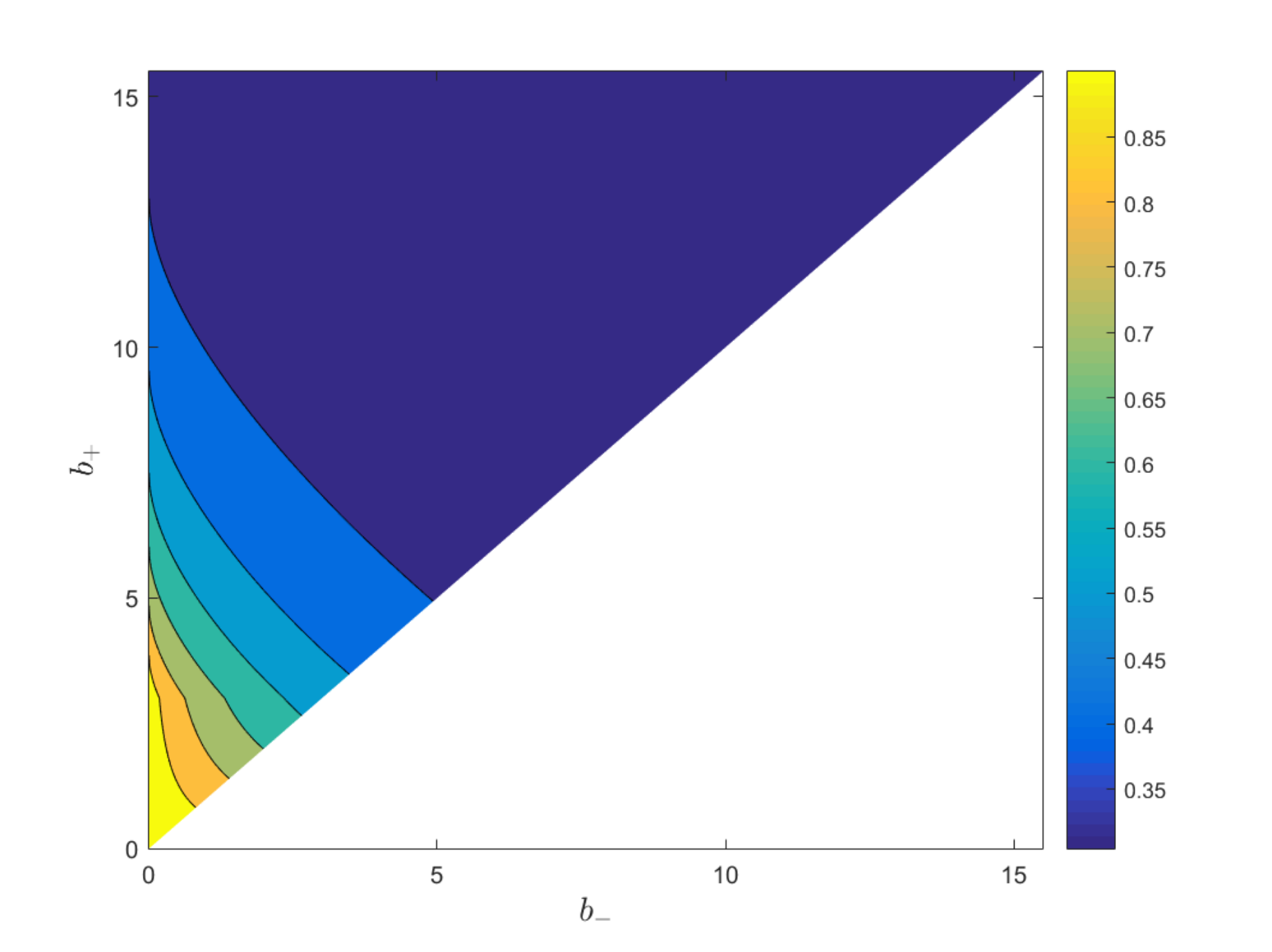}
	\caption{$x=3$, $\bar{K}_x=0.303$}
   \end{subfigure}
\begin{subfigure}[b]{.485\linewidth}
      \includegraphics[width=1\textwidth]{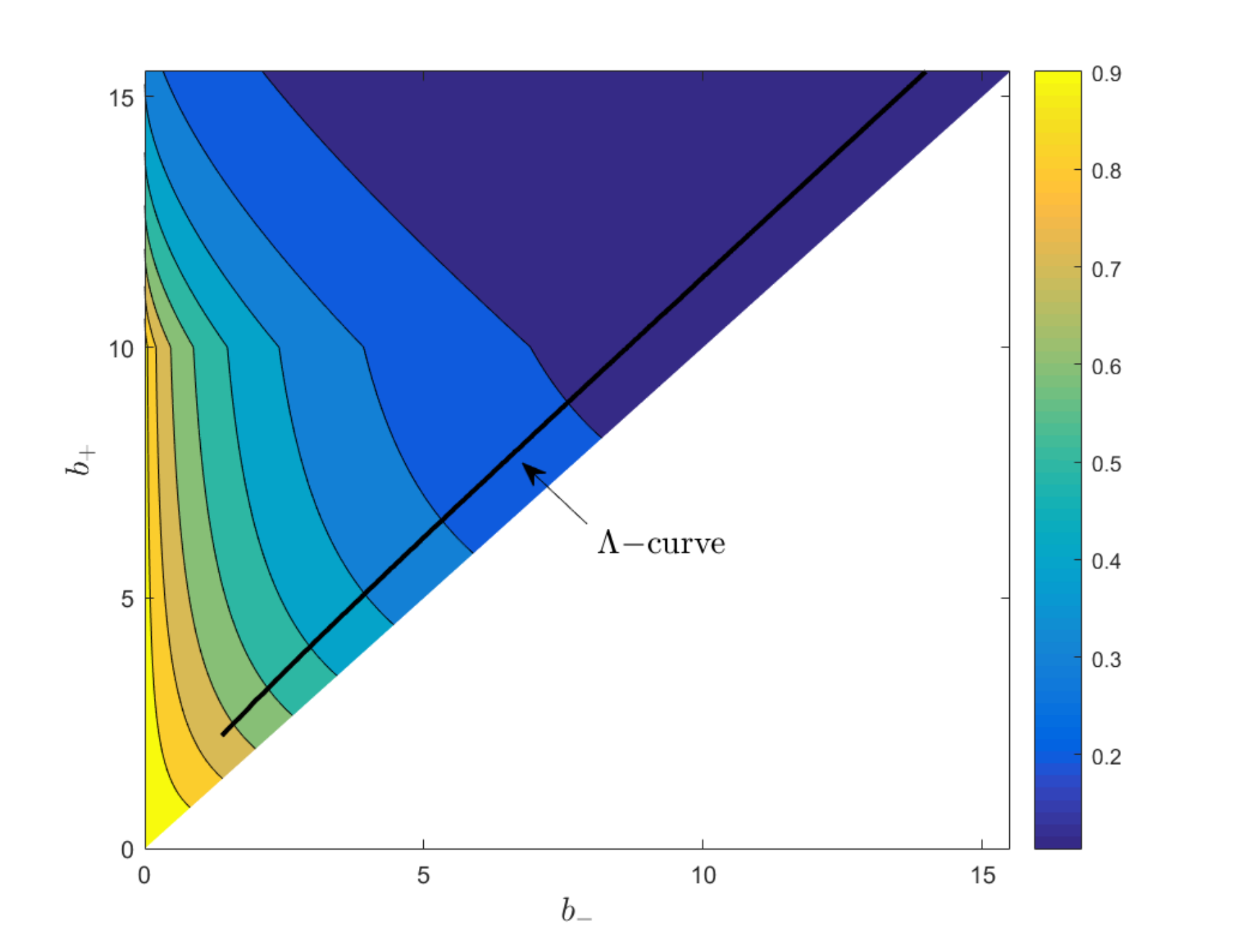}
	\caption{$x=10$, $\bar{K}_x=0.097$}
   \end{subfigure}
      \caption{Contour plots of $\Psi_x(b_-,b_+)$.}\label{figPhiTran}
\end{figure}

\begin{figure}[t!]
      \includegraphics[width=0.5\textwidth]{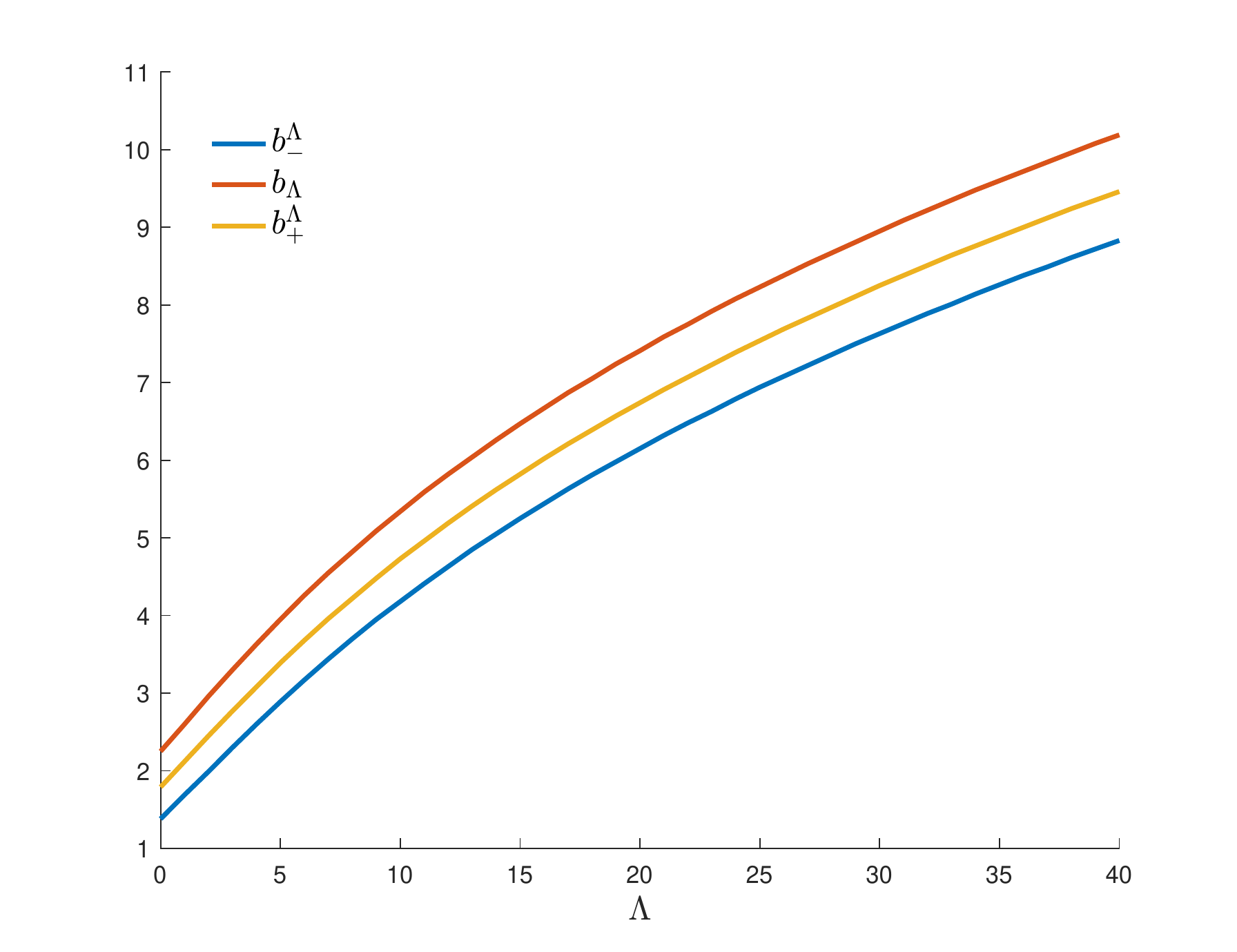}
      \caption{Optimal pair $(b_-^{\Lambda},b_+^{\Lambda})$ and $b_{\Lambda}$.}\label{figoptb}
\end{figure}

Finally, Figure \ref{figoptL} shows the value of $\Lambda^*$ from Proposition \ref{optLambda} for different values of $K$ as a function of $x$, and the value functions for the unconstrained and constrained problems.

\begin{figure}[t!]
   \begin{subfigure}[b]{.495\linewidth}
      \includegraphics[width=1\textwidth]{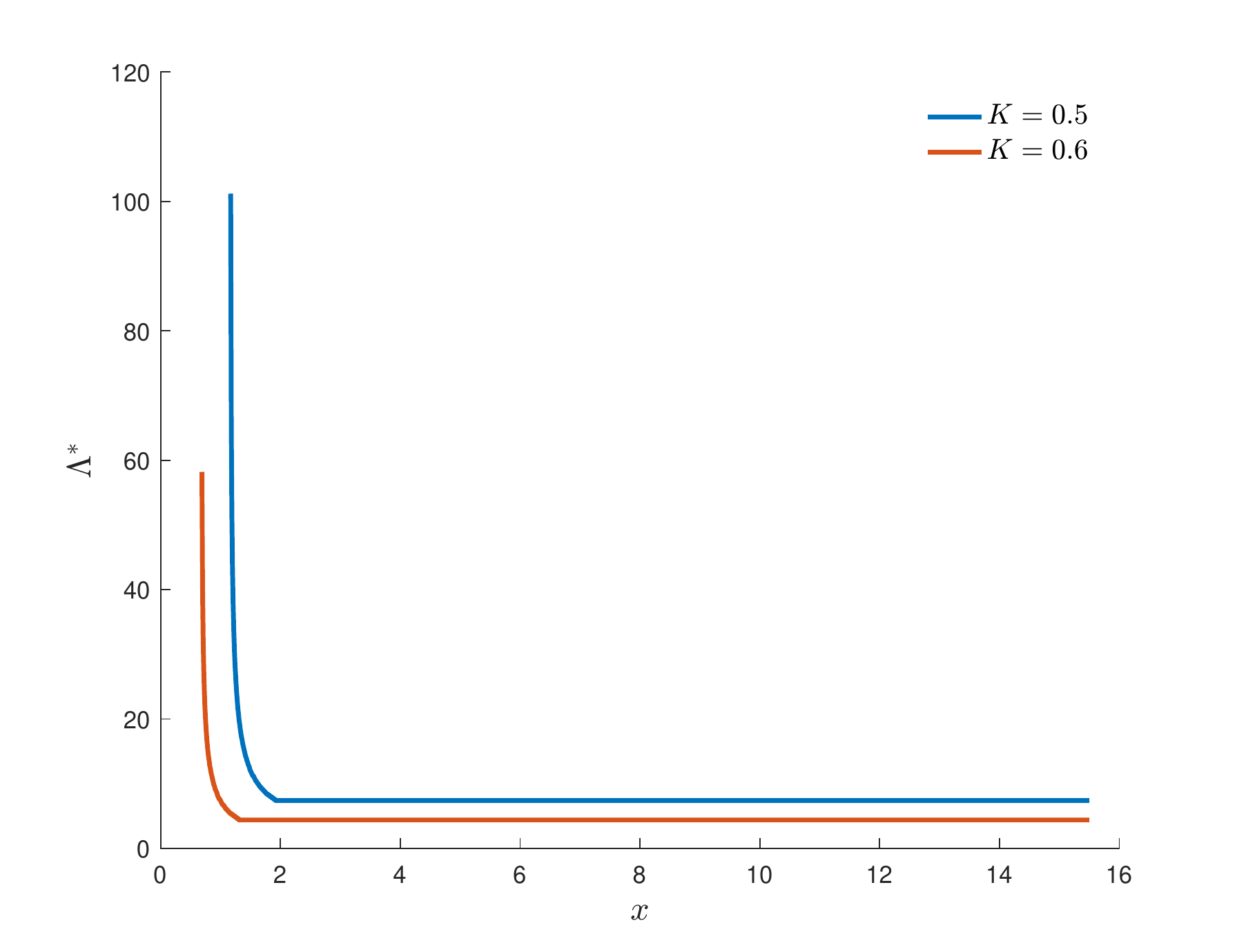}
	\caption{$\Lambda^*$}
   \end{subfigure}
\begin{subfigure}[b]{.485\linewidth}
      \includegraphics[width=1\textwidth]{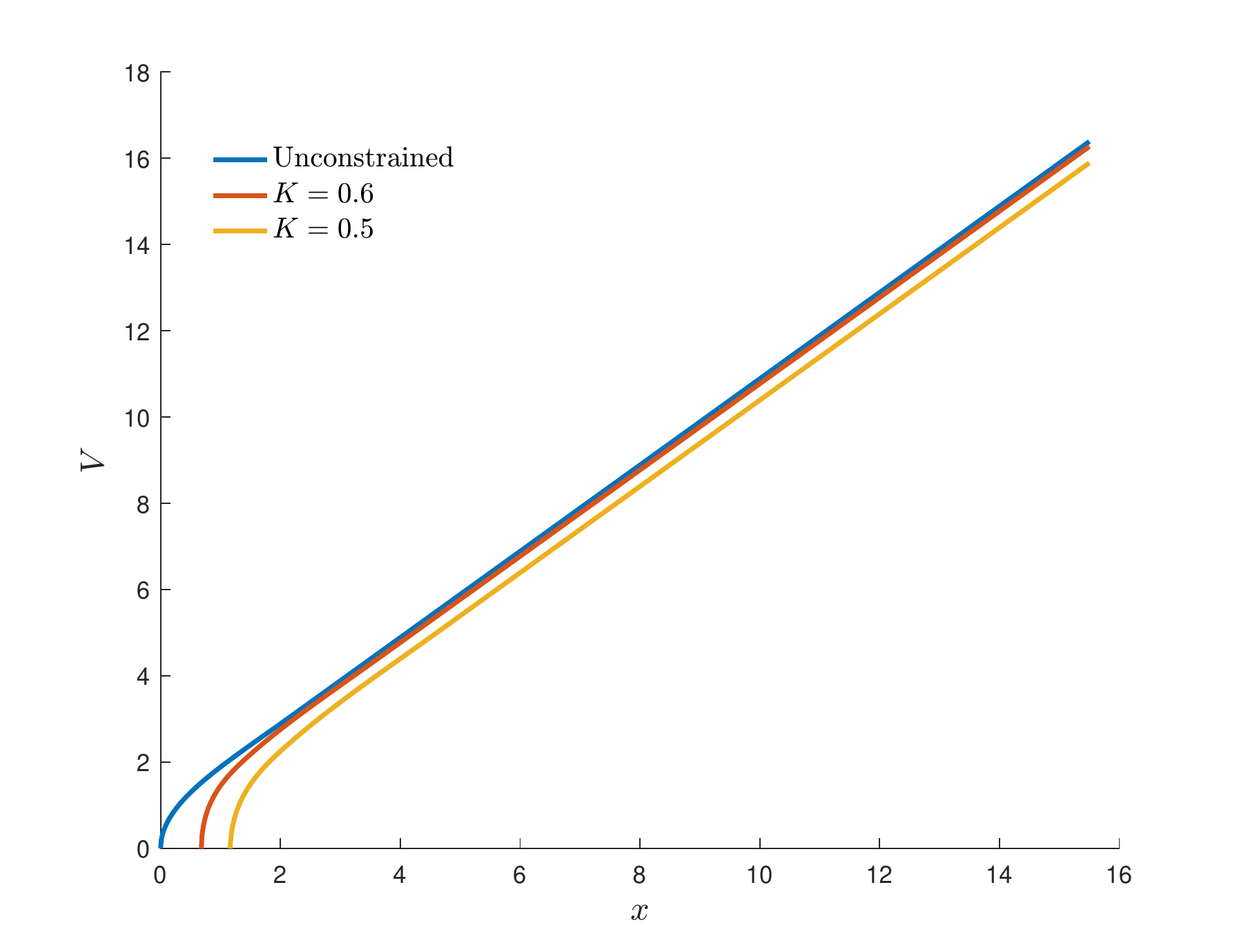}
	\caption{$V$}
   \end{subfigure}
      \caption{$\Lambda^*$ and value function for the constrained problem with different values of $K$.}\label{figoptL}
\end{figure}
\end{ex}

\section{Conclusions and future work}

In the framework of the classical dividend problem there exists a trade-off between stability and profitability. We were able to continue the work started in \cite{HJ15} in order to solve the optimal dividend problem subject to a constraint in the time of ruin. Using the fundamental tool of scale functions and fluctuation theory we improved the previous result for spectrally one-sided L\'evy processes, and included the case of fixed transaction cost.

New questions arise from this work. The first one would be if the same results hold for band strategies instead of barrier strategies. This would probably require to know in advance the number of bands. Now, if we continue working with barrier strategies, to find other constraints that fit with the tools developed in this work is another challenging question. Finally, an interesting question is the existence of a Hamilton-Jacobi-Bellman-like equation that characterizes the value function of the constrained problem, this could open the door for a new theory for constrained stochastic optimal control. 

\section*{Acknowledgments}
Mauricio Junca was supported by Universidad de los Andes under the Grant Fondo de Apoyo a Profesores Asistentes (FAPA). Harold Moreno-Franco acknowledges financial support from HSE, which was given within the framework of a subsidy granted  to the HSE by the government of the Russian Federation for the implementation of the Global Competitiveness Program. 
\bibliographystyle{alpha} 
\bibliography{ref}

\begin{thebibliography}{BKY14b}

\bibitem[AKP04]{Avram04}
F.~Avram, A.~Kyprianou, and M.~Pistorious.
\newblock Exit problems for spectrally negative {L}\'evy processes and
  application to (canadized) russian options.
\newblock {\em The Annals of Applied Probability}, 14(1):2115--238, 2004.

\bibitem[AM05]{azcuemuler2005}
Pablo Azcue and Nora Muler.
\newblock Optimal reinsurance and dividend distribution policies in the
  {C}ram\'er-{L}undberg model.
\newblock {\em Mathematical Finance}, 15(2):261--308, 03 2005.

\bibitem[APP07]{AvPaPi07}
Florin Avram, Zbigniew Palmowski, and Martin Pistorius.
\newblock On the optimal dividend problem for a spectrally negative {L}\'evy
  process.
\newblock {\em The Annals of Applied Probability}, 17(1):156--180, 2007.

\bibitem[APP15]{avram2015}
F.~Avram, Z.~Palmowski, and M.~R. Pistorius.
\newblock On {G}erber-{S}hiu functions and optimal dividend distribution for a
  {L}\'evy risk process in the presence of a penalty function.
\newblock {\em Ann. Appl. Probab.}, 25(4):1868--1935, 2015.

\bibitem[AT97]{asta}
S{\o}ren Asmussen and Michael Taksar.
\newblock Controlled diffusion models for optimal dividend pay-out.
\newblock {\em Insurance Math. Econom.}, 20(1):1--15, 1997.

\bibitem[BKY14a]{KyYa14}
E~Bayraktar, A.E. Kyprianou, and K.~Yamazaki.
\newblock On optimal dividends in the dual problem.
\newblock {\em ASTIN Bulletin}, 43(3):359--372, 2014.

\bibitem[BKY14b]{BayraktarImpdual}
Erhan Bayraktar, Andreas~E. Kyprianou, and Kazutoshi Yamazaki.
\newblock Optimal dividends in the dual model under transaction costs.
\newblock {\em Insurance: Mathematics and Economics}, 54:133--143, 2014.

\bibitem[DF57]{Definetti}
Bruno De~Finetti.
\newblock Su un’impostazion alternativa dell teoria collecttiva del rischio.
\newblock In {\em Transactions of the XVth International Congress of
  Actuaries}, number~2, pages 433--443, 1957.

\bibitem[Ger72]{Gerber72}
Hans~U. Gerber.
\newblock Games of economic survival with discrete- and continuous-income
  processes.
\newblock {\em Operations Research}, 20(1):pp. 37--45, 1972.

\bibitem[Gra15]{Grandits}
Peter Grandits.
\newblock An optimal consumption problem in finite time with a constraint on
  the ruin probability.
\newblock {\em Finance and Stochastics}, 19(4):791--847, 2015.

\bibitem[Hip03]{Hipp03}
Christian Hipp.
\newblock Optimal dividend payment under a ruin constraint: Discrete time and
  state space.
\newblock {\em Blätter der DGVFM}, 26(2):255--264, 2003.

\bibitem[HJ15]{HJ15}
Camilo Hernandez and Mauricio Junca.
\newblock Optimal dividend payments under a constraint on the time value of
  ruin: Exponential case.
\newblock {\em Insurance: Mathematics and Economics}, 65(1):136--142, 2015.

\bibitem[KKR13]{KKRivero2013}
A.~Kuznetsov, A.~E. Kyprianou, and V.~Rivero.
\newblock The theory of scale functions for spectrally negative {L}\'evy
  processes.
\newblock {\em Levy Matters II}, Vol. 2061:97--186, 2013.

\bibitem[KRS10]{KyprianouRS10}
A.E. Kyprianou, V.~Rivero, and R.~Song.
\newblock Smoothness and convexity of scale functions with applications to de
  {F}inetti’s control problem.
\newblock {\em Journal of Theoretical Probabilty}, 23(2):547--564, 2010.

\bibitem[Kyp14]{kyprianou2014}
Andreas~E. Kyprianou.
\newblock {\em Fluctuations of {L}\'evy Processes with Applications}.
\newblock Universitext. Springer-Verlag Berlin Heidelberg, 2014.

\bibitem[Loe08]{Loeffen082}
R.~Loeffen.
\newblock On the optimality of the barrier strategy in de {F}inetti's problem
  for spectrally negative {L}\'evy processes.
\newblock {\em The Annals of Applied Probability}, 18(5):1669--1680, 2008.

\bibitem[Loe09a]{LoeffenTrans}
R.~L. Loeffen.
\newblock An optimal dividends problem with transaction costs for spectrally
  negative {L}\'evy processes.
\newblock {\em Insurance: Mathematics and Economics}, 45(1):41--48, 2009.

\bibitem[Loe09b]{Loeffen08}
Ronnie Loeffen.
\newblock An optimal dividends problem with a terminal value for spectrally
  negative {L}\'evy processes with a completely monotone jump density.
\newblock {\em Journal of Applied Probability}, 46(1):85--98, 2009.

\bibitem[LR10]{Loeffen10}
Ronnie Loeffen and Jean-Fran{\c{c}}ois Renaud.
\newblock De {F}inetti's optimal dividends problem with an affine penalty
  function at ruin.
\newblock {\em Insurance: Mathematics and Economics}, 46(1):98--108, 2010.

\bibitem[Pau03]{Jostein03}
Jostein Paulsen.
\newblock Optimal dividend payouts for diffusions with solvency constraints.
\newblock {\em Finance and Stochastics}, 7(4):457--473, 09 2003.

\bibitem[Sch02]{schmidli2002}
Hanspeter Schmidli.
\newblock On minimizing the ruin probability by investment and reinsurance.
\newblock {\em The Annals of Applied Probability}, 12(3):890--907, 08 2002.

\bibitem[Sch08]{Schmidli}
Hanspeter Schmidli.
\newblock {\em Stochastic Control in Insurance}.
\newblock Probability and Its Applications. Springer, 2008.

\bibitem[Sur08]{surya2008}
B.~A. Surya.
\newblock Evaluating scale functions of spectrally negative {L}\'evy processes.
\newblock {\em J. Appl. Probab.}, 45(1):135--149, 03 2008.

\bibitem[TA07]{ThonAlbr}
Stefan Thonhauser and Hansjorg Albrecher.
\newblock Dividend maximization under consideration of the time value of ruin.
\newblock {\em Insurance: Mathematics and Economics}, 41(1):163--184, 2007.

\bibitem[Tak00]{tak}
Michael~I. Taksar.
\newblock Optimal risk and dividend distribution control models for an
  insurance company.
\newblock {\em Math. Methods Oper. Res.}, 51(1):1--42, 2000.

\bibitem[YW13]{Yin}
Chuancun Yin and Yuzhen Wen.
\newblock Optimal dividend problem with a terminal value for spectrally
  positive {L}\'evy processes.
\newblock {\em Insurance: Mathematics and Economics}, 53(3):769--773, 2013.

\end{thebibliography}

\end{document}